\documentclass[10pt]{amsart}
\usepackage{amssymb}
\usepackage{amscd}
\usepackage[all]{xy}

\numberwithin{equation}{section}
\newcounter{TmpEnumi}

\def\today{\number\day\space\ifcase\month\or   January\or February\or
   March\or April\or May\or June\or   July\or August\or September\or
   October\or November\or December\fi\   \number\year}

\theoremstyle{definition}
\newtheorem{thm}{Theorem}[section]
\newtheorem{lem}[thm]{Lemma}
\newtheorem{prp}[thm]{Proposition}
\newtheorem{dfn}[thm]{Definition}
\newtheorem{cor}[thm]{Corollary}

\newtheorem{rmk}[thm]{Remark}
\newtheorem{ntn}[thm]{Notation}
\newtheorem{exa}[thm]{Example}
\newtheorem{pbm}[thm]{Problem}

\newcommand{\beq}{\begin{equation}}
\newcommand{\eeq}{\end{equation}}
\newcommand{\beqr}{\begin{eqnarray*}}
\newcommand{\eeqr}{\end{eqnarray*}}
\newcommand{\bal}{\begin{align*}}
\newcommand{\eal}{\end{align*}}
\newcommand{\bei}{\begin{itemize}}
\newcommand{\eei}{\end{itemize}}
\newcommand{\limi}[1]{\lim_{{#1} \to \infty}}

\newcommand{\af}{\alpha}
\newcommand{\bt}{\beta}
\newcommand{\gm}{\gamma}
\newcommand{\dt}{\delta}
\newcommand{\ep}{\varepsilon}
\newcommand{\zt}{\zeta}
\newcommand{\et}{\eta}
\newcommand{\ch}{\chi}
\newcommand{\io}{\iota}

\newcommand{\ld}{\lambda}
\newcommand{\sm}{\sigma}
\newcommand{\kp}{\kappa}
\newcommand{\ph}{\varphi}
\newcommand{\ps}{\psi}
\newcommand{\rh}{\rho}
\newcommand{\om}{\omega}
\newcommand{\ta}{\tau}

\newcommand{\Gm}{\Gamma}

\newcommand{\Z}{{\mathbb{Z}}}

\newcommand{\C}{{\mathbb{C}}}
\newcommand{\N}{{\mathbb{Z}}_{> 0}}
\newcommand{\Nz}{{\mathbb{Z}}_{\geq 0}}

\newcommand{\aGA}{\af \colon G \to \Aut (A)}

\newcommand{\GAa}{(G, A, \af)}
\newcommand{\GBb}{(G, B, \bt)}

\newcommand{\CGAa}{C^* (G, A, \af)}

\newcommand{\ga}{$G$-algebra}

\newcommand{\id}{{\mathrm{id}}}

\newcommand{\sa}{{\mathrm{sa}}}
\newcommand{\spec}{{\mathrm{sp}}}

\newcommand{\card}{{\mathrm{card}}}
\newcommand{\Aut}{{\mathrm{Aut}}}
\newcommand{\Ad}{{\mathrm{Ad}}}

\newcommand{\dirlim}{\varinjlim}

\newcommand{\andeqn}{\,\,\,\,\,\, {\mbox{and}} \,\,\,\,\,\,}

\newcommand{\ts}[1]{{\textstyle{#1}}}
\newcommand{\ds}[1]{{\displaystyle{#1}}}

\newcommand{\sssum}[2]{{\ts{ {\ds{\sum}}_{#1}^{#2} }}}

\newcommand{\wolog}{without loss of generality}

\newcommand{\ifo}{if and only if}

\newcommand{\ca}{C*-algebra}
\newcommand{\uca}{unital C*-algebra}

\newcommand{\hm}{homomorphism}

\newcommand{\fd}{finite dimensional}

\newcommand{\hsa}{hereditary subalgebra}

\newcommand{\pj}{projection}
\newcommand{\mops}{mutually orthogonal \pj s}

\newcommand{\mvnt}{Murray-von Neumann equivalent}

\newcommand{\ct}{continuous}
\newcommand{\cfn}{continuous function}

\newcommand{\eqv}{equivariant}
\newcommand{\ehm}{equivariant homomorphism}
\newcommand{\sj}{semiprojective}
\newcommand{\eqsj}{equivariantly semiprojective}
\newcommand{\eqrsj}{equivariantly conditionally semiprojective}

\newcommand{\Lem}[1]{Lemma~\ref{#1}}
\newcommand{\Def}[1]{Definition~\ref{#1}}

\newcommand{\OA}[1]{{\mathcal{O}}_{#1}}

\newcommand{\OI}{\OA{\infty}}

\title[Equivariant semiprojectivity]{Equivariant semiprojectivity}   

\author{N.~Christopher Phillips}

\date{20~December 2011}

\address{Department of Mathematics, University  of Oregon,
       Eugene OR 97403-1222, USA,
    and
       Research Institute for Mathematical Sciences,
  Kyoto University,
  Kitashirakawa-Oiwakecho, Sakyo-ku, Kyoto
  606-8502, Japan.}

\email[]{ncp@darkwing.uoregon.edu}

\thanks{This material is based upon work supported by the
  US National Science Foundation under Grants
  DMS-0701076 and DMS-1101742.
  It was also partially supported by the Centre de Recerca
  Matem\`{a}tica (Barcelona) through a research visit conducted
  during 2011.}

\begin{document}

\begin{abstract}
We define equivariant semiprojectivity for C*-algebras
equipped with actions of compact groups.
We prove that the following examples are equivariantly semiprojective:
\begin{itemize}
\item
Arbitrary finite dimensional C*-algebras
with arbitrary actions of compact groups.
\item
The Cuntz algebras~${\mathcal{O}}_d$
and extended Cuntz algebras~$E_d,$ for finite~$d,$
with quasifree actions of compact groups.
\item
The Cuntz algebra~${\mathcal{O}}_{\infty}$
with any quasifree action of a finite group.
\end{itemize}
For actions of finite groups,
we prove that equivariant semiprojectivity
is equivalent to
a form of equivariant stability of generators and relations.
We also prove that if $G$ is finite,
then $C^* (G)$ is graded semiprojective.
\end{abstract}

\maketitle

\indent
Semiprojectivity has become recognized as the ``right'' way
to formulate many approximation results
in \ca s.
The standard reference is Loring's book~\cite{Lr}.
The formal definition and its basic properties are in
Chapter~14 of~\cite{Lr},
but much of the book is really about
variations on semiprojectivity.
Also see the more recent survey article~\cite{Bl7}.
There has been considerable work since then.

In this paper,
we introduce an equivariant version of semiprojectivity
for \ca s with actions of compact groups.
(The definition makes sense for actions of arbitrary groups,
but seems likely to be interesting only when the group is compact.)
The motivation for the definition and our choice of results
lies in applications which will be presented elsewhere.
We prove that arbitrary actions of compact groups
on \fd\  \ca s are \eqsj,
that quasifree actions of compact groups
on the Cuntz algebras~${\mathcal{O}}_d$
and the extended Cuntz algebras~$E_d,$ for finite~$d,$
are \eqsj,
and that quasifree actions of finite groups
on~${\mathcal{O}}_{\infty}$
are \eqsj.
We also give, for finite group actions,
an equivalent condition for equivariant semiprojectivity
in terms of equivariant stability of generators and relations.

In a separate paper~\cite{PST},
we prove the following results relating
equivariant semiprojectivity
and ordinary semiprojectivity.
If $G$ is finite
and $\GAa$ is equivariantly semiprojective,
then $\CGAa$ is semiprojective.
If $G$ is compact and second countable,
$A$~is separable,
and $\GAa$ is equivariantly semiprojective,
then $A$ is semiprojective.
Examples show that finiteness of~$G$ is necessary
in the first statement,
and that neither result has a converse.

We do not address equivariant semiprojectivity of actions
on Cuntz-Krieger algebras,
on $C ([0, 1]) \otimes M_n,$ $C (S^1) \otimes M_n,$
or dimension drop intervals
(except for a result for $C (S^1)$
which comes out of our work on quasifree
actions; see Remark~\ref{R-CS1}),
or on $C^* (F_n).$
We presume that suitable actions
on these algebras are \eqsj,
but we leave investigation of them for future work.

We also presume that there are interesting and useful
equivariant analogs of
weak stability of relations (Definition 4.1.1 of~\cite{Lr}),
weak semiprojectivity (Definition 4.1.3 of~\cite{Lr}),
projectivity (Definition 10.1.1 of~\cite{Lr}),
and liftability of relations (Definition 8.1.1 of~\cite{Lr}).
Again, we do not treat them.
(Equivariant projectivity will be discussed in~\cite{PST}.)

Finally, we point out work in the commutative case.
It is well known that $C (X)$ is semiprojective in the
category of commutative \ca s
\ifo\  $X$ is an absolute neighborhood retract.
Equivariant absolute neighborhood retracts have
a significant literature;
as just three examples, we refer to the papers \cite{Jw},
\cite{AE}, and~\cite{An}.
(I am grateful to Adam P. W. S{\o}rensen for calling my attention
to the existence of this work.)

This paper is organized as follows.
Section~\ref{Sec:Basic}
contains the definition of equivariant semiprojectivity,
some related definitions, and the proofs of some basic results.

Section~\ref{Sec:FdIsESj} contains the proof
that any action of a compact group
on a \fd\  \ca\  %
is \eqsj.
As far as we can tell,
traditional functional calculus methods
(a staple of~\cite{Lr})
are of little use here.
We use instead an iterative method for showing that
approximate homomorphisms from compact groups
are close to true \hm s.
For a compact group~$G,$
we also prove that equivariant semiprojectivity is preserved
when tensoring with any \fd\  \ca\  with any action of~$G.$

In Section~\ref{Sec:QFCuntz},
we prove that quasifree actions
of compact groups on the Cuntz algebra ${\mathcal{O}}_d$
and the extended Cuntz algebras~$E_d,$
for $d$ finite, are \eqsj.
We use an iterative method similar to that
used for actions of \fd\  \ca s,
but this time applied to cocycles.
Section~\ref{Sec:QFOI} extends the result to quasifree actions
on~${\mathcal{O}}_{\infty},$
but only for finite groups.
The method is that of Blackadar~\cite{Bl7},
but a considerable amount of work needs to be done to set this up.
We do not know whether the result extends
to quasifree actions of general compact groups
on~${\mathcal{O}}_{\infty}.$

In Section~\ref{Sec:QqSR},
we show that the universal \ca\  given by a bounded finite
equivariant set of generators and relations
is \eqsj\  \ifo\  the relations are equivariantly stable.
This is the result which enables most of the current applications
of equivariant semiprojectivity.
It is important for these applications
that an approximate representation is only required
to be approximately equivariant.
We give one application here:
we show that in the Rokhlin and tracial Rokhlin properties
for an action of a finite group,
one can require that the Rokhlin projections
be exactly permuted by the group.

Section~\ref{Sec:GdSj}
contains a proof that for a finite group~$G,$
the algebra $C^* (G),$
with its natural $G$-grading,
is graded semiprojective.
This result uses the same machinery as the proof
that actions on \fd\  \ca s are \eqsj.
We do not go further in this direction,
but this result suggests that there is a much more
general theory,
perhaps of equivariant semiprojectivity for actions
of \fd\  quantum groups.

I am grateful to
Bruce Blackadar, Ilijas Farah,
Adam P.\  W.\  S{\o}rensen, and Hannes Thiel
for valuable discussions.
I also thank the Research Institute for Mathematical Sciences
of Kyoto University for its support through a visiting professorship.

\section{Definitions and basic results}\label{Sec:Basic}

The following definition is the analog of
Definition 14.1.3 of~\cite{Lr}.

\begin{dfn}\label{D:EqSj}
Let $G$ be a topological group,
and let $\GAa$ be a unital \ga.
We say that $\GAa$ is {\emph{equivariantly semiprojective}}
if whenever $(G, C, \gm)$ is a unital \ga,
$J_0 \subset J_1 \subset \cdots$ are $G$-invariant ideals in~$C,$
$J = {\overline{\bigcup_{n = 0}^{\infty} J_n}},$
\[
\kp \colon C \to C / J,
\,\,\,\,\,\,
\kp_n \colon C \to C / J_n,
\andeqn
\pi_n \colon C / J_n \to C / J
\]
are the quotient maps,
and $\ph \colon A \to C / J$
is a unital equivariant \hm,
then there exist~$n$
and a unital equivariant \hm\  $\ps \colon A \to C / J_n$
such that $\pi_n \circ \ps = \ph.$

When no confusion can arise, we say that $A$ is \eqsj,
or that $\af$ is \eqsj.
\end{dfn}

Here is the diagram:
\[
\xymatrix{
& C \ar@{-->}[d]^{\kp_n} \ar@/^2pc/[dd]^{\kp} \\
& C / J_n \ar@{-->}[d]^{\pi_n} \\
A \ar[r]_-{\ph} \ar@{-->}[ru]^{\ps} & C / J.
}
\]
The solid arrows are given,
and $n$ and $\ps$ are supposed to exist which make the
diagram commute.

We suppose that \Def{D:EqSj} is probably only
interesting when $G$ is compact.
Blackadar has shown that,
in the nonunital category,
the trivial action of~$\Z$ on~$\C$ is not \eqsj~\cite{BlX}.
This is equivalent to saying the the trivial action of $\Z$
on $\C \oplus \C$ is not \eqsj\  in the sense defined here.
In the unital category, the trivial action of
any group on~$\C$ is \eqsj\  for trivial reasons,
but there are no other known examples of \eqsj\  actions of
noncompact groups.

We will also need the following form of
equivariant semiprojectivity for \hm s.
Our definition is {\emph{not}} an analog of
the definition of semiprojectivity for \hm s given
before Lemma 14.1.5 of~\cite{Lr}.
Rather, it is related to the second step
in the idea of two step lifting
as in Definition 8.1.6 of~\cite{Lr},
with the caveat that lifting as there corresponds to
projectivity rather than semiprojectivity of a \ca.
It is the equivariant version of a special case of
conditional semiprojectivity as in Definition~5.11 of~\cite{ELP}.

\begin{dfn}\label{D:EqsjForHom}
Let $G$ be a topological group,
let $\GAa$ and $\GBb$ be unital \ga s,
and let $\om \colon A \to B$ be a unital equivariant \hm.
We say that $\om$ is {\emph{equivariantly conditionally semiprojective}}
if whenever $(G, C, \gm)$ is a unital \ga,
$J_0 \subset J_1 \subset \cdots$ are $G$-invariant ideals in~$C,$
$J = {\overline{\bigcup_{n = 0}^{\infty} J_n}},$
\[
\kp \colon C \to C / J,
\,\,\,\,\,\,
\kp_n \colon C \to C / J_n,
\andeqn
\pi_n \colon C / J_n \to C / J
\]
are the quotient maps,
and $\ld \colon A \to C$
and $\ph \colon B \to C / J$
are unital equivariant \hm s
such that $\kp \circ \ld = \ph \circ \om,$
then there exist~$n$
and a unital equivariant \hm\  $\ps \colon B \to C / J_n$
such that
\[
\pi_n \circ \ps = \ph
\andeqn
\kp_n \circ \ld = \ps \circ \om.
\]
\end{dfn}

Here is the diagram:
\[
\xymatrix{
& & C \ar@{-->}[d]^{\kp_n} \ar@/^2pc/[dd]^{\kp} \\
& & C / J_n \ar@{-->}[d]^{\pi_n} \\
A \ar[r]_{\om} \ar[rruu]^{\ld}
    & B \ar[r]_-{\ph} \ar@{-->}[ru]^{\ps} & C / J.
}
\]
The part of the diagram with the solid arrows is assumed to commute,
and $n$ and $\ps$ are supposed to exist which make the
whole diagram commute.

\begin{rmk}\label{R:EqSj}
\begin{enumerate}
\item\label{R:EqSj:U}
\Def{D:EqSj} is stated for the category of unital \ga s.
Without the group,
a unital \ca\  is semiprojective in the unital category
\ifo\  it is semiprojective in the nonunital category.
(See Lemma 14.1.6 of~\cite{Lr}.)
The same is surely true here,
and should be essentially immediate from what we do,
but we don't need it and do not give a proof.
\item\label{R:EqSj:Lift}
In the situations of \Def{D:EqSj} and \Def{D:EqsjForHom},
we say that $\ps$ {\emph{equivariantly lifts}}~$\ph.$
\item\label{R:EqSj:Nt}
In proofs, we will adopt the standard notation
$\pi_{n, m} \colon C / J_m \to C / J_n,$
for $m, n \in \N$ with $n \geq m,$
for the maps between the different quotients
implicit in \Def{D:EqSj} and \Def{D:EqsjForHom}.
Thus $\pi_n \circ \pi_{n, m} = \pi_m$
and $\pi_{n, m} \circ \pi_{m, l} = \pi_{n, l}$
for suitable choices of indices.
We further let $\gm^{(n)} \colon G \to \Aut (C / J_n)$
and $\gm^{(\infty)} \colon G \to \Aut (C / J)$
be the induced actions on the quotients.
\end{enumerate}
\end{rmk}

\begin{lem}\label{L:SjTrtve}
Let $G$ be a topological group,
let $\GBb$ be a unital \ga,
let $A \subset B$ be a unital $G$-invariant subalgebra,
and let $\om \colon A \to B$ be the inclusion.
If $A$ is \eqsj\  and $\om$ is \eqrsj, then $B$ is \eqsj.
\end{lem}

\begin{proof}
Let the notation be as in Definition~\ref{D:EqSj}
and Remark~\ref{R:EqSj}(\ref{R:EqSj:Nt}).
Suppose $\ph \colon B \to C / J$ is an equivariant unital \hm.
Then equivariant semiprojectivity of~$A$
implies that there are $n_0$ and an equivariant
unital \hm\  $\ld \colon A \to C / J_{n_0}$
such that $\pi_{n_0} \circ \ld = \ph |_A.$
Now apply equivariant conditional semiprojectivity of~$\om,$
with $C / J_{n_0}$ in place of $C$
and the ideals $J_n / J_0,$ for $n \geq n_0,$
in place of the ideals $J_n.$
We obtain $n \geq n_0$ and an equivariant unital \hm\  %
$\ps \colon B \to C / J_n$
such that $\pi_n \circ \ps = \ph$
(and also $\pi_{n, n_0} \circ \ld = \ps |_A$).
\end{proof}

\begin{ntn}\label{N:FixPt}
Let $\GAa$ be a \ga.
We denote by $A^G$ the fixed point algebra
\[
A^G = \{ a \in A \colon {\mbox{$\af_g (a) = a$ for all $g \in G$}} \}.
\]
In case of ambiguity of the action, we write $A^{\af}.$

Further, if $\GBb$ is another \ga\  and $\ph \colon A \to B$
is an equivariant \hm,
then $\ph$ induces a \hm\  from $A^G$ to $B^G,$
which we denote by $\ph^G.$
\end{ntn}

We need the following two easy lemmas.

\begin{lem}\label{L:FixedPtQuot}
Let $G$ be a compact group,
and let $(G, C, \gm)$ be a \ga.
Let $J \subset C$ be a $G$-invariant ideal.
Then the obvious map $\rh \colon C^G / J^G \to C / J$
is injective and has range exactly $(C / J)^G.$
\end{lem}

\begin{proof}
Injectivity is immediate from the relation $J \cap C^G = J^G.$
It is obvious that $\rh ( C^G / J^G ) \subset (C / J)^G.$
For the reverse inclusion,
let $x \in (C / J)^G.$
Let $\pi \colon C \to C / J$ be the quotient map.
Choose $c \in C$ such that $\pi (c) = x.$
Let $\mu$ be Haar measure on~$G,$
normalized so that $\mu (G) = 1.$
Set
\[
a = \int_G \gm_g (c) \, d \mu (g).
\]
Then $a \in C^G$ and $\pi (a) = x.$
Therefore $a + J^G \in C^G / J^G$ and $\rh (a + J^G) = x.$
\end{proof}

\begin{lem}\label{L-FixedPtLim}
Let $G$ be a compact group,
and let $(G, A, \af)$ be a \ga.
Let $A_0 \subset A_1 \subset \cdots$
be an increasing sequence of $G$-invariant
subalgebras of $A$ such that
${\overline{\bigcup_{n = 0}^{\infty} A_n}} = A.$
Then $A^G = {\overline{\bigcup_{n = 0}^{\infty} A_n^G}}.$
\end{lem}

\begin{proof}
It is clear that
${\overline{\bigcup_{n = 0}^{\infty} A_n^G}} \subset A^G.$
For the reverse inclusion,
let $a \in A^G$ and let $\ep > 0.$
Choose $n$ and $x \in A_n$ such that $\| x - a \| < \ep.$
Let $\mu$ be Haar measure on~$G,$
normalized such that $\mu (G) = 1.$
Then $b = \int_G \gm_g (x) \, d \mu (g)$ is in $A_n^G$
and satisfies $\| b - a \| < \ep.$
\end{proof}

Now we are ready to prove equivariant semiprojectivity of
some \ga s.

\begin{lem}\label{L-TrOnSbgp}
Let $G$ be a compact group,
let $N \subset G$ be a closed normal subgroup,
and let $\rh \colon G \to G / N$ be the quotient map.
Let $A$ be a \uca,
and let $\af \colon G / N \to \Aut (A)$
be an \eqsj\  action of $G / N$ on~$A.$
Then $(G, \, A, \, \af \circ \rh)$ is \eqsj.
\end{lem}

\begin{proof}
We claim that
there is an action ${\overline{\gm}} \colon G / N \to \Aut (C^N)$
such that for $g \in G$ and $c \in C^N$
we have ${\overline{\gm}}_{g N} (c) = \gm_g (c).$
One only needs to check that ${\overline{\gm}}$ is well defined,
which is easy.

Let the notation be as in Definition~\ref{D:EqSj}
and Remark~\ref{R:EqSj}(\ref{R:EqSj:Nt}).
Then $\ph (A) \subset (C / J)^N,$
which by \Lem{L:FixedPtQuot} is the same as $C^N / J^N.$
Let $\ph_0 \colon A \to C^N / J^N$ be the corestriction.
\Lem{L-FixedPtLim} implies
$J^N = {\overline{\bigcup_{n = 0}^{\infty} J_n^N}},$
so semiprojectivity of $(G / N, \, A, \, \af)$ provides $n$
and a unital $G / N$-\eqv\  \hm\  $\ps_0 \colon A \to C^N / J_n^N$
which lifts~$\ph_0.$
We take $\ps$ to be the following composition,
in which the middle map comes from \Lem{L:FixedPtQuot}
and the last map is the inclusion:
\[
A \stackrel{\ps_0}{\longrightarrow} C^N / J_n^N
  \longrightarrow (C / J_n)^N
  \longrightarrow C / J_n.
\]
Then $\ps$ is $G$-equivariant and lifts~$\ph.$
\end{proof}

\begin{cor}\label{C-TrIsSj}
Let $G$ be a compact group,
let $A$ be a \uca,
and let $\io \colon G \to \Aut (A)$ be the trivial action of $G$ on~$A.$
If $A$ is semiprojective,
then $(G, A, \io)$ is \eqsj.
\end{cor}

\begin{proof}
In \Lem{L-TrOnSbgp},
take $N = G.$
\end{proof}

\begin{cor}\label{C:SjIncId}
Let $G$ be a compact group,
and let $\GAa$ be a unital \ga.
Then $A$ is \eqsj\  \ifo\  the inclusion of $\C \cdot 1$ in $A$
is \eqrsj\  in the sense of \Def{D:EqsjForHom}.
\end{cor}

\begin{proof}
The subalgebra $\C \cdot 1$ is \eqsj\  by Corollary~\ref{C-TrIsSj},
so we may apply Lemma~\ref{L:SjTrtve}.
\end{proof}

\begin{prp}\label{P:RSjDSum}
Let $G$ be a compact group,
and let $\big( \big( G, A_k, \af^{(k)} \big) \big)_{k = 1}^m$
be a finite collection of \eqsj\  unital \ga s.
Suppose that $l \in \{0, 1, \ldots, m - 1 \}.$
Set $A = \left( \bigoplus_{k = 1}^l A_k \right) \oplus \C$ and
set $B = \bigoplus_{k = 1}^m A_k,$
with the obvious direct sum actions $\af \colon G \to \Aut (A)$
(with $G$ acting trivially on~$\C$)
and $\bt \colon G \to \Aut (B).$
Define $\om \colon A \to B$ by
\[
\om (a_1, a_2, \ldots, a_l, \ld)
 = \big( a_1, \, a_2, \, \ldots, \, a_l, \, \ld \cdot 1_{A_{l + 1}}, \,
       \ld \cdot 1_{A_{l + 2}}, \, \ldots, \, \ld \cdot 1_{A_{m}} \big)
\]
for
\[
a_1 \in A_1,
\,\,\,\,\,\,\
a_2 \in A_2,
\,\,\,\,\,\,\
\ldots,
\,\,\,\,\,\,\
a_l \in A_l,
\andeqn
\ld \in \C.
\]
Then $\om$ is \eqrsj.
\end{prp}

\begin{proof}
Let the notation be as in Definition~\ref{D:EqsjForHom}
and Remark~\ref{R:EqSj}(\ref{R:EqSj:Nt}).
For $k = 1, 2, \ldots, l$
let $e_k \in A$ be the identity of the summand $A_k \subset A,$
and for $k = 1, 2, \ldots, m$
let $f_k \in B$ be the identity of the summand $A_k \subset B.$
Set $q = 1 - \sum_{k = 1}^l \mu (e_k).$
Let $P \subset B$ be the subalgebra generated by
$f_{l + 1}, f_{l + 2}, \ldots, f_m.$
Then $P$ is semiprojective and $G$ acts trivially on it.
Therefore Corollary~\ref{C-TrIsSj} provides $n_0$ and a
unital equivariant \hm\  %
$\ps_0 \colon P \to q C q / q J_{n_0} q$
such that $\pi_{n_0} \circ \ps_0 = \ph |_P.$
For $k = l + 1, \, l + 2, \, \ldots, \, m,$
set $p_k = \ps_0 (e_k).$
Use equivariant semiprojectivity of~$A_k,$
with $p_k (C / J_{n_0}) p_k$ in place of~$C$
and with $p_k (J_n / J_{n_0}) p_k$ in place of $J_n$
(for $n \geq n_0$) to find $n_k \geq n_0$ and a
unital equivariant lifting
\[
\ps_k \colon A_k
   \to \pi_{n_k, n_0} (p_k) (C / J_{n_k}) \pi_{n_k, n_0} (p_k)
\]
of $\ph |_{A_k}.$
Define $n = \max (n_1, n_2, \ldots, n_m),$
and define $\ps \colon A \to C / J_n$ by
\[
\ps (a_1, a_2, \ldots, a_m)
 = (\kp_n \circ \mu) (a_1, a_2, \ldots, a_l)
       + \sum_{k = l + 1}^m \pi_{n, n_k} ( \ps_k (a_k)).
\]
Then $\ps$ is an equivariant lifting of~$\ph.$
\end{proof}

\begin{cor}\label{P:SjDSum}
Let $G$ be a compact group,
and let $\big( \big( G, A_k, \af^{(k)} \big) \big)_{k = 1}^m$
be a finite collection of \eqsj\  unital \ga s.
Then $A = \bigoplus_{k = 1}^m A_k,$
with the direct sum action $\af \colon G \to \Aut (A),$
is \eqsj.
\end{cor}

\begin{proof}
Proposition~\ref{P:RSjDSum} (with $l = 0$) implies that the
unital inclusion of $\C$ in $A$ is \eqrsj,
so Corollary~\ref{C:SjIncId} implies that $A$ is \eqsj.
\end{proof}

We can use traditional methods to give an example
of a nontrivial action which is equivariantly semiprojective.
This result will be superseded
in Theorem~\ref{T:FDEqSj} below,
using more complicated methods,
so the proof here will be sketchy.

\begin{prp}\label{P:TransSj}
Let $G$ be a finite cyclic group.
Let $G$ act on $C (G)$ by the translation action,
$\ta_g (a) (h) = a (g^{-1} h)$ for $g, h \in G$ and $a \in C (G).$
Then $(G, C (G), \ta)$ is \eqsj.
\end{prp}

\begin{proof}
Let the notation be as in Definition~\ref{D:EqSj}
and Remark~\ref{R:EqSj}(\ref{R:EqSj:Nt}).

Take
$G = \Z / d \Z
   = \big\{ 1, \, e^{2 \pi i / d}, \,  e^{4 \pi i / d},
      \,  \ldots, \,
        e^{2 (d - 1) \pi i / d} \big\}
   \subset S^1.$
Let $u$ be the inclusion of $G$ in~$S^1,$
which we regard as a unitary in $C (G).$
Then $u$ generates $C (G)$
and $\ta_{\ld} (u) = \ld^{-1} u$ for $\ld \in G.$
Therefore it suffices to find $n$ and a unitary $z \in C / J_n$
such that $\pi_n (z) = \ph (u),$
$\spec (z) \subset G,$
and
$\gm^{(n)}_{\ld} (z) = \ld^{-1} z$ for all $\ld \in G.$

Since $C (G)$ is semiprojective (in the nonequivariant sense),
there are $n_0$ and a unitary $v_0 \in C / J_{n_0}$
such that $\pi_n (v_0) = \ph (u)$ and $v_0^d = 1.$
Moreover, for all $\ld \in G,$
we have
\[
\limi{n}
 \big\| \pi_{n, n_0} \big( \gm^{(n_0)}_{\ld} (v_0) \big)
        - \ld^{-1} v_0 \big\|
= 0.
\]

Choose $\ep > 0$ such that
$\ep < \tfrac{1}{2} \big| 1 - e^{\pi i / d} \big|,$
and such that
whenever $B$ is a \uca\  and $b \in B$ satisfies $\| b - 1 \| < \ep,$
then
$\big\| b (b^* b)^{-1/2} - 1 \big\|
 < \tfrac{1}{2} \big| 1 - e^{\pi i / d} \big|.$
Choose $n$ so large that $v = \pi_{n, n_0} (v_0)$ satisfies
$\big\| \gm^{(n)}_{\ld} (v) - \ld^{-1} v \big\| < \ep$
for all $\ld \in G.$
Define $a \in C / J_n$ by
\[
a = \frac{1}{d} \sum_{\ld \in G} \ld \gm^{(n)}_{\ld} (v).
\]
Then one checks that
$\gm^{(n)}_{\ld} (a) = \ld^{-1} a$ for all $\ld \in G$
and that $\| a - v \| < \ep < 1,$
so $a$ is invertible.
Set $w = a (a^* a)^{-1/2},$
and check that $\gm^{(n)}_{\ld} (w) = \ld^{-1} w$ for all $\ld \in G.$
A calculation,
using the choice of~$\ep,$
shows that
$\| w - v \| < \tfrac{1}{2} \big| 1 - e^{\pi i / d} \big|.$
So $e^{\pi i / d} G \cap \spec (w) = \varnothing.$
Let $f \colon S^1 \setminus e^{\pi i / d} G \to S^1$
be the function determined by
$g (e^{i t}) = e^{2 \pi i k / d}$
when $t \in \big( \tfrac{2 k - 1}{d}, \, \tfrac{2 k + 1}{d} \big).$
Then
$f (\ld \zt) = \ld f (\zt)$ for all $\ld \in G$
and $\zt \in S^1 \setminus e^{\pi i / d} G,$
and $f$ is \ct\  on $\spec (w).$
Define $z = f (w).$
The verification that $z$ satisfies the required conditions
is a calculation.
\end{proof}

\section{Equivariant semiprojectivity
 of finite dimensional C*-algebras}\label{Sec:FdIsESj}

\indent
The main result of this section is that actions
of compact groups on \fd\  \ca s are \eqsj.

The main technical tool is a method
for replacing approximate homomorphisms to
unitary groups by nearby exact homomorphisms,
in such a way as to preserve properties such as being
equivariant.
(In Section~\ref{Sec:GdSj},
we will also need to preserve the property of being graded.)
The method used here has been discovered twice before,
in Theorem~3.8 of~\cite{GKR2}
(most of the work is in Section~4 of~\cite{GKR1},
but the result in~\cite{GKR1} uses the wrong metric on the groups)
and in Theorem~1 of~\cite{Kz}.
It is not clear from either of these proofs that the
additional properties we need are preserved.
We will instead follow the proofs of Theorem~5.13
and Proposition~5.14 of~\cite{AGG}.
(We are grateful to Ilijas Farah for pointing out these references.)

\begin{ntn}\label{N:UGp}
For a unital \ca~$A,$
we let $U (A)$ denote the unitary group of~$A.$
\end{ntn}

The following lemmas give an estimate
whose proof is omitted in~\cite{AGG}.
We will need this estimate again,
in the proof of Lemma~\ref{L:1StepCcy} below.
(We don't get quite the same estimate as implied in~\cite{AGG}.)

\begin{lem}\label{L-IntegralEst}
Let $\Gm$ be a compact group with normalized Haar measure~$\mu.$
Let $A$ be a \uca.
Suppose $r \in \big[ 0, \tfrac{1}{2} \big],$
and let $u \colon \Gm \to U (A)$ be a \cfn\  such that
$\| u (g) - 1 \| \leq r$ for all $g \in G.$
Then
\[
\left\| \int_{\Gm} u (g)  \, d \mu (g)
    - \exp \left( \int_{\Gm} \log (u (g) ) \, d \mu (g) \right) \right\|
   \leq \frac{5 r^2}{2 (1 - 2 r)}
\]
and
\[
\left\| \int_{\Gm} u (g)  \, d \mu (g)  \right\| \leq 1.
\]
\end{lem}

\begin{proof}
The second statement is obvious.

For the first, we require the following estimates
(compare with Lemma 5.15 of~\cite{AGG}):
for $u \in U (A)$ with $\| u - 1 \| < 1,$
we have
\begin{equation}\label{Eq:LogEst}
\big\| \log (u) - (u - 1) \big\|
  \leq \frac{\| u - 1 \|^2}{2 \big(1 - \| u - 1 \|  \big)},
\end{equation}
and for $a \in A$ with $\| a \| < 1,$
we have
\begin{equation}\label{Eq:ExpEst}
\big\| \exp (a) - (1 + a) \big\|
  \leq \frac{\| a \|^2}{2 \big(1 - \| a \|  \big)}.
\end{equation}
Both are obtained from power series:
\[
\big\| \log (u) - (u - 1) \big\|
  \leq \sum_{n = 2}^{\infty} \frac{\| u - 1 \|^n}{n}
  \leq \frac{1}{2} \sum_{n = 2}^{\infty} \| u - 1 \|^n
\]
and
\[
\big\| \exp (a) - (1 + a) \big\|
  \leq \sum_{n = 2}^{\infty} \frac{\| a \|^n}{n!}
  \leq \frac{1}{2} \sum_{n = 2}^{\infty} \| a \|^n.
\]

Apply~(\ref{Eq:LogEst}) to the condition
$\| u (g) - 1 \| \leq r$ and integrate, getting
\begin{equation}\label{Eq:IntLogEst}
\left\| \int_{\Gm} \log (u (g)) \, d \mu (g)
  - \int_{\Gm} u (g)  \, d \mu (g) - 1 \right\|
   \leq \frac{r^2}{2 (1 - r)}.
\end{equation}
Since $r \leq \frac{1}{2},$
we also get
\[
\left\| \int_{\Gm} \log (u (g)) \, d \mu (g)  \right\|
 \leq \frac{r^2}{2 (1 - r)} + \int_{\Gm} \| u (g) - 1 \| \, d \mu (g)
 \leq 2 r.
\]
We therefore get, integrating and using~(\ref{Eq:ExpEst}),
\[
\left\| \exp \left( \int_{\Gm} \log (u (g) ) \, d \mu (g) \right)
   - \int_{\Gm} \log (u (g) ) \, d \mu (g) - 1 \right\|
 \leq  \frac{(2 r)^2}{2 (1 - 2 r)}.
\]
Combining this estimate with~(\ref{Eq:IntLogEst})
gives
\[
\left\| \exp \left( \int_{\Gm} \log (u (g) ) \, d \mu (g) \right)
   - \int_{\Gm} u (g)  \, d \mu (g) \right\|
 \leq \frac{r^2}{2 (1 - r)} + \frac{(2 r)^2}{2 (1 - 2 r)}
 \leq \frac{5 r^2}{2 (1 - 2 r)},
\]
as desired.
\end{proof}

\begin{lem}\label{L:OneStep}
Let $\Gm$ be a compact group with normalized Haar measure~$\mu.$
Let $A$ be a \uca.
Suppose $r \in \big[ 0, \tfrac{1}{5} \big],$
and let $\rh \colon \Gm \to U (A)$
be a \cfn\  such that for all $g, h \in \Gm$ we have
\[
\| \rh (g h) - \rh (g) \rh (h) \| \leq r.
\]
For $g \in \Gm$ define
\[
\sm (g)
 = \exp \left( \int_{\Gm}
          \log \Big( \rh (k)^* \rh (k g) \rh (g)^* \Big) \, d \mu (k)
       \right) \rh (g).
\]
Then $\sm$ is a \cfn\  from $\Gm$ to $U (A)$ which satisfies
\[
\| \sm (g h) - \sm (g) \sm (h) \| \leq 17 r^2
\andeqn
\| \sm (g) - \rh (g) \| \leq 2 r
\]
for all $g, h \in \Gm.$
\end{lem}

\begin{proof}
For $g \in \Gm,$
define
\[
\sm_0 (g) = \int_{\Gm} \rh (k)^* \rh (k g) \, d \mu (k).
\]
The first part of the proof of Proposition~5.14 of~\cite{AGG}
shows that for $g, h \in \Gm,$ we have
\[
\| \sm_0 (g h) - \sm_0 (g) \sm_0 (h) \| \leq 2 r^2
\andeqn
\| \sm_0 (g) - \rh (g) \| \leq r.
\]
The rest of the proof in~\cite{AGG} uses a Lie algebra valued
logarithm, called ``$\ln$'' there.
We replace statements in~\cite{AGG} involving the Lie algebra
of the codomain with the use of the logarithm
coming from holomorphic functional calculus.
Rewriting
\[
\sm_0 (g) =
  \left( \int_{\Gm} \rh (k)^* \rh (k g) \rh (g)^* \, d \mu (k)
       \right) \rh (g)
\]
and applying \Lem{L-IntegralEst},
we get
\[
\| \sm (g) - \sm_0 (g) \| \leq \frac{5 r^2}{2 (1 - 2 r)}
    \leq 5 r^2
    \leq r
\]
for all $g \in \Gm.$
This implies $\| \sm (g) - \rh (g) \| \leq 2 r$ for all $g \in \Gm,$
which is the second of the required estimates.
Clearly $\| \sm (g) \| \leq 1$ for all $g \in \Gm.$
\Lem{L-IntegralEst} implies $\| \sm_0 (g) \| \leq 1$
for all $g \in \Gm.$
Therefore
\begin{align*}
\lefteqn{\| \sm (g h) - \sm (g) \sm (h) \|}
  \\
& \leq \| \sm (g h) - \sm (g h) \|
           +  \| \sm (g) - \sm (g) \|
           +  \| \sm (h) - \sm (h) \|
           +  \| \sm_0 (g h) - \sm_0 (g) \sm_0 (h) \|
   \\
& \leq 5 r^2 + 5 r^2 + 5 r^2 + 2 r^2
  = 17 r^2.
\end{align*}
This is the first of the required estimates.
\end{proof}

\begin{lem}\label{L:A7}
Let $\Gm$ be a compact group with normalized Haar measure~$\mu.$
Let $A$ and $B$ be \uca s,
and let $\kp \colon A \to B$ be a unital \hm.
Suppose $0 \leq r < \frac{1}{17},$
and let $\rh_0 \colon \Gm \to U (A)$
be a \ct\  map such that for all $g, h \in \Gm,$
we have
\[
\| \rh_0 (g h) - \rh_0 (g) \rh_0 (h) \| \leq r
\andeqn
(\kp \circ \rh_0) (g h) = (\kp \circ \rh_0) (g) (\kp \circ \rh_0) (h).
\]
Inductively define functions $\rh_m \colon \Gm \to A$
by (following \Lem{L:OneStep})
\[
\rh_{m + 1} (g)
 = \exp \left( \int_{\Gm}
       \log \Big( \rh_m (k)^* \rh_m (k g) \rh_m (g)^* \Big) \, d \mu (k)
       \right) \rh_m (g)
\]
for $g \in \Gm.$
Then for every $m \in \N$
the function $\rh_m$ is a well defined \cfn\  %
from $\Gm$ to $U (A)$ such that $\kp \circ \rh_m = \kp \circ \rh_0.$
Moreover, the functions $\rh_m$ converge uniformly to
a \ct\  \hm\  $\rh \colon \Gm \to U (A)$ such that
\[
\sup_{g \in \Gm} \| \rh (g) - \rh_0 (g) \| \leq \frac{2 r}{1 - 17 r}
\andeqn
\kp \circ \rh = \kp \circ \rh_0.
\]
\end{lem}

\begin{proof}
We claim that for all $m \in \Nz,$
the function $\rh_m$ is well defined,
\ct, take values in $U (A),$
and satisfies $\kp \circ \rh_m = \kp \circ \rh_0,$
and that for $g, h \in \Gm$ we have
\begin{equation}\label{Eq:Better}
\| \rh_m (g h) - \rh_m (g) \rh_m (h) \| \leq r (17 r)^m
\end{equation}
and
\begin{equation}\label{Eq:Conv}
\| \rh_m (g) - \rh_{m - 1} (g) \| \leq 2 r (17 r)^{m - 1}.
\end{equation}

The proof of the claim is by induction on~$m.$
The case $m = 1$ is \Lem{L:OneStep} and $r \leq \tfrac{1}{5}.$
Assume the result is known for~$m.$
Since the estimates (\ref{Eq:Better}) and~(\ref{Eq:Conv})
hold for~$m,$
and by \Lem{L:OneStep} and because $r (17 r)^m < r \leq \tfrac{1}{5},$
the function $\rh_{m + 1}$ is well defined,
\ct, take values in $U (A),$
and for $g, h \in \Gm$ we have
\[
\| \rh_{m + 1} (g) - \rh_m (g) \|
  \leq 2 r (17 r)^m
\]
and, also using $17 r < 1$ at the last step,
\[
\| \rh_{m + 1} (g h) - \rh_{m + 1} (g) \rh_{m + 1} (h) \|
   \leq 17 \big( r (17 r)^m \big)^2
   = r (17 r)^{m + 1} (17 r)^m
   < r (17 r)^{m + 1}.
\]
It remains to prove that
$\kp \circ \rh_{m + 1} = \kp \circ \rh_0.$
Let $g \in \Gm.$
Using $\kp \circ \rh_m = \kp \circ \rh_0$ at the second step
and the fact that $\kp \circ \rh_0$ is a \hm\  at the
last step, we get
\begin{align*}
\lefteqn{\kp \left( \exp \left( \int_{\Gm}
      \log \Big( \rh_m (k)^* \rh_m (k g) \rh_m (g)^* \Big) \, d \mu (k)
       \right) \right)}
       \\
& \hspace*{2em} {\mbox{}}
  = \exp \left( \int_{\Gm}
         \log \Big( (\kp \circ \rh_m) (k)^* (\kp \circ \rh_m) (k g)
                   (\kp \circ \rh_m) (g)^* \Big) \, d \mu (k)
       \right)
   \\
& \hspace*{2em} {\mbox{}}
  = \exp \left( \int_{\Gm}
         \log \Big( (\kp \circ \rh_0) (k)^* (\kp \circ \rh_0) (k g)
                   (\kp \circ \rh_0) (g)^* \Big) \, d \mu (k)
       \right)
  = 1.
\end{align*}
Therefore
$\kp (\rh_{m + 1} (g)) = \kp (\rh_{m} (g)) = \kp (\rh_0 (g)).$
This completes the induction, and proves the claim.

The estimate~(\ref{Eq:Conv}) implies that there is a
\cfn\  $\rh \colon \Gm \to U (A)$ such that $\rh_m \to \rh$
uniformly,
and in fact for $g \in \Gm$ we have
\[
\| \rh (g) - \rh_0 (g) \|
  \leq \sum_{m = 1}^{\infty} 2 r (17 r)^{m - 1}
  = \frac{2 r}{1 - 17 r}.
\]
The estimate~(\ref{Eq:Better}) and convergence imply that
$\rh$ is a \hm.
Continuity of~$\kp$ implies that $\kp \circ \rh = \kp \circ \rh_0.$
\end{proof}

The following proposition is a variant of the fact that two
close \hm s from a \fd\  \ca\  are unitarily equivalent.

\begin{prp}\label{P:CloseImpUE}
Let $\Gm$ be a compact group,
let $A$ and $B$ be \uca s,
and let $\kp \colon A \to B$ be a unital \hm.
Let $\rh, \sm \colon \Gm \to U (A)$
be two \ct\  \hm s
such that
\[
\| \rh (g) - \sm (g) \| < 1
\andeqn
\kp \circ \rh (g) = \kp \circ \sm (g)
\]
for all $g \in \Gm.$
Then there exists a unitary $u \in A$
such that $u \rh (g) u^* = \sm (g)$ for all $g \in \Gm,$
and such that $\kp (u) = 1.$
\end{prp}

\begin{proof}
Let $\mu$ be normalized Haar measure on~$\Gm.$
Define
\[
a = \int_{\Gm} \sm (h)^* \rh (h) \, d \mu (h).
\]
For $g \in \Gm$ we get, changing variables at the second step,
\begin{equation}\label{Eq:CjA}
a \rh (g)
 = \int_{\Gm} \sm (h)^* \rh (h g) \, d \mu (h)
 = \int_{\Gm} \sm \big( h g^{-1} \big)^* \rh (h) \, d \mu (h)
 = \sm (g) a.
\end{equation}
Since $\| \sm (h)^* \rh (h) - 1 \| < 1$ for all $h \in \Gm,$
we have $\| a - 1 \| < 1.$
Therefore
$u = a (a^* a)^{- 1/2}$ is a well defined unitary in~$A.$
Taking adjoints in~(\ref{Eq:CjA}),
we get $a^* \sm (g) = \rh (g) a^*$
for all $g \in \Gm,$
so $a^* a$ commutes with $\rh (g).$
Thus $(a^* a)^{- 1/2}$ commutes with $\rh (g).$
Applying~(\ref{Eq:CjA}) again,
we get $u \rh (g) = \sm (g) u$ for all $g \in \Gm.$

The hypotheses imply that $\kp (a) = 1,$
so also $\kp (u) = 1.$
\end{proof}

\begin{thm}\label{T:FDEqSj}
Let $\af \colon G \to \Aut (A)$ be an action of a compact group~$G$
on a \fd\  \ca~$A.$
Then $(G, A, \af)$ is equivariantly semiprojective.
\end{thm}

\begin{proof}
Set $\ep_0 = \frac{1}{6 \cdot 34},$
and choose $\ep > 0$ such that $\ep \leq \ep_0$ and
such that whenever $A$ is a \uca,
$u \in U (A),$
and $a \in A$ satisfies $\| a - u \| < \ep,$
then we have
$\big\| a (a^* a)^{-1/2} - u \big\| < \ep_0.$

Let the notation be as in Definition~\ref{D:EqSj}
and Remark~\ref{R:EqSj}(\ref{R:EqSj:Nt}).
Let $\ph \colon A \to C / J$ be a unital \ehm.
Since \fd\  \ca s are \sj,
there exist $n_0$ and a unital \hm\  %
(not necessarily \eqv) $\ps_0 \colon A \to C / J_{n_0}$
which lifts~$\ph.$

For $n \geq n_0,$
define $f_n \colon G \times U (A) \to [0, \infty)$
by
\[
f_n (g, x)
 = \big\| \pi_{n, n_0} \big( \ps_0 ( \af_g (x))
                        - \gm_g^{(n_0)} (\ps_0 (x)) \big) \big\|
\]
for $g \in G$ and $x \in U (A).$
The functions $f_n$ are \ct\  and satisfy
\[
f_{n_0} \geq f_{n_0 + 1} \geq f_{n_0 + 2} \geq \cdots.
\]
Using $J = {\overline{\bigcup_{n = 1}^{\infty} J_n}}$
at the first step
and equivariance of $\ph$ at the second step,
for $g \in G$ and $x \in U (A)$ we have
\[
\limi{n} f_n (g, x)
  = \| \ph (\af_g (x)) - \gm_g ( \ph (x)) \|
  = 0.
\]
Since $G \times U (A)$ is compact,
Dini's Theorem (Proposition~11 in Chapter~9 of~\cite{Ry})
implies that $f_n \to 0$ uniformly.
Therefore there exists $n \geq n_0$ such that
for all $g \in G$ and $x \in U (A),$
we have
\[
\big\|
 \pi_{n, n_0} \big( \ps_0 ( \af_g (x))
                        - \gm_g^{(n_0)} (\ps_0 (x)) \big) \big\|
  < \ep.
\]
Set $\ps_1 = \pi_{n, n_0} \circ \ps.$
Then this estimate becomes
\begin{equation}\label{Eq:TEqSj-U}
\big\| \ps_1 ( \af_g (x)) - \gm_g^{(n)} (\ps_1 (x)) \big\|
  < \ep.
\end{equation}
for every $g \in G$ and $x \in U (A).$

Let $\nu$ be normalized Haar measure on~$G.$
For $x \in A$ define
\[
T (x) = \int_G \big( \gm_h^{(n)} \circ \ps_1 \circ \af_h^{-1} \big) (x)
        \, d \nu (h).
\]
Then for $g \in G$ we have
\begin{align*}
\gm_g^{(n)} (T (x))
& = \int_G \big( \gm_{g h}^{(n)} \circ \ps_1 \circ \af_h^{-1} \big) (x)
        \, d \nu (h)
         \\
& = \int_G
       \big( \gm_h^{(n)} \circ \ps_1 \circ \af_{g^{-1} h}^{-1} \big) (x)
        \, d \nu (h)
  = T ( \af_g (x) ).
\end{align*}
So $T$ is equivariant.
Also, since $\pi_n \circ \ps_1 = \ph$ is \eqv,
we have $\pi_n (T (x)) = \ph (x)$ for all $x \in A.$
It follows from~(\ref{Eq:TEqSj-U})
that $\| T (x) - \ps_1 (x) \| < \ep$ for all $x \in U (A).$
Since $\ep < 1,$
we may define $\rh_0 \colon U (A) \to U (C / J_n)$ by
\[
\rh_0 (x) = T (x) \big( T (x)^* T (x) \big)^{-1/2}
\]
for $x \in U (A).$
Then
\[
\gm_g^{(n)} (\rh_0 (x)) = \rh_0 (\af_g (x) )
\andeqn
\pi_n ( \rh_0 (x)) = \ph (x)
\]
for all $g \in G$ and $x \in U (A).$
By the choice of~$\ep,$
we have
$\| \rh_0 (x) - T (x) \| < \ep_0,$
whence
\[
\| \rh_0 (x) - \ps_1 (x) \|
 \leq \| \rh_0 (x) - T (x) \| + \| T (x) - \ps_1 (x) \|
 < \ep_0 + \ep
 \leq 2 \ep_0.
\]
Let $x, y \in U (A).$
Since
\[
\ps_1 (x), \, \ps_1 (y) \in U (C / J_n)
\andeqn
\ps_1 (x y) = \ps_1 (x) \ps_1 (y),
\]
it follows that
\[
\| \rh_0 (x y) - \rh_0 (x) \rh_0 (y) \| < 6 \ep_0.
\]

Let $\mu$ be normalized Haar measure on the
compact group $U (A).$
Inductively define functions $\rh_m \colon \Gm \to U (C / J_m)$
by (following \Lem{L:A7})
\[
\rh_{m + 1} (x)
 = \rh_m (x) \exp \left( \int_{U (A)}
      \log \Big( \rh_m (x)^* \rh_m (x y) \rh_m (y)^* \Big) \, d \mu (y)
       \right)
\]
for $x \in U (A).$
Since $6 \ep_0 = \frac{1}{34} < \frac{1}{17},$
Lemma~\ref{L:A7} implies that each function $\rh_m$
is a well defined \cfn\  from $U (A)$ to $U (C / J_n)$
and that $\rh (x) = \limi{m} \rh_m (x)$
defines a \ct\  \hm\  from $U (A)$ to $U (C / J_n)$
satisfying
\[
\| \rh (x) - \rh_0 (x) \|
 \leq \frac{2 \cdot 6 \ep_0}{1 - 17 \cdot 6 \ep_0}
 = \frac{2}{17}
\andeqn
\pi_n ( \rh (x)) = \ph (x)
\]
for all $x \in U (A).$
Since \hm s respect functional calculus,
an induction argument shows that
\[
\gm_g^{(n)} (\rh_m (x)) = \rh_m (\af_g (x) )
\]
for all $m \in \Nz,$ $g \in G,$ and $x \in U (A).$
Therefore also
\begin{equation}\label{Eq:TEqSj-Eq}
\gm_g^{(n)} (\rh (x)) = \rh (\af_g (x) )
\end{equation}
for all $g \in G$ and $x \in U (A).$

For $x \in U (A)$ we have
\[
\| \rh (x) - \ps_1 (x) \|
  \leq \| \rh (x) - \rh_0 (x) \| + \| \rh_0 (x) - \ps_1 (x) \|
  < \tfrac{2}{17} + 6 \ep_0
  < 1.
\]
Since $\pi_n ( \rh (x)) = \ph (x) = \pi_n ( \ps_1 (x))$
for $x \in U (A),$
and since $U (A)$ is compact,
Proposition~\ref{P:CloseImpUE}
provides a unitary $w \in C / J_n$
such that $\pi_n (w) = 1$ and such that
$w \ps_1 (x) w^* = \rh (x)$ for all $x \in U (A).$
Define a \hm\  $\ps \colon A \to C / J_n$
by $\ps (a) = w \ps_1 (x) w^*$ for $a \in A.$
Then $\ps$ lifts $\ph$ because $\pi_n (w) = 1.$
Furthermore, $\ps$ is equivariant by~(\ref{Eq:TEqSj-Eq})
and because $U (A)$ spans~$A.$
\end{proof}

As an immediate application,
one can require that the projections in the definitions of
the Rokhlin and tracial Rokhlin properties for finite groups
be exactly orthogonal and exactly permuted by the group action,
rather than merely being approximately permuted by the group action.
We postpone the proof until after discussion equivariant stability
of relations.
See Proposition~\ref{P-RkP}
and Proposition~\ref{P-TRP}.

We can now show that tensoring with \fd\  \ga s
preserves equivariant semiprojectivity.
The proof is essentially due to
Adam P.\  W.\  S{\o}rensen, and Hannes Thiel
and we are grateful to them for their permission to include it here.
We begin with a lemma.

\begin{lem}\label{L-SurjComm}
Let $A_1$ and $A_2$ be \uca s,
and let $\ph \colon A_1 \to A_2$ be a surjective \hm.
Let $F$ be a \fd\  \ca,
and let $\ld_1 \colon F \to A_1$ be a unital \hm.
Set $\ld_2 = \ph \circ \ld_1.$
For $s = 1, 2$ define
\[
B_s = \big\{ a \in A_s \colon
    {\mbox{$a$ commutes with $\ld (x)$ for all $x \in F$}} \big\}.
\]
Then $\ph |_{B_1}$ is a surjective \hm\  from $B_1$ to~$B_2.$
\end{lem}

\begin{proof}
There are $n, \, r (1), \, r (2), \, \ldots, \, r (n) \in \N$
such that $F = \bigoplus_{l = 1}^n F_k$ and $F_l \cong M_{r (l)}$
for $l = 1, 2, \ldots, n.$
Let $\big( e_{j, k}^{(l)} \big)_{j, k = 1}^{r (l)}$
be a system of matrix units for~$F_l.$

For $s = 1, 2$ define
$E_s \colon A_s \to A_s$ by
\begin{equation}\label{Eq:L-SurjComm}
E_s (a) = \sum_{l = 1}^n \sum_{k = 1}^{r (l)}
  \ld_s \big( e_{k, 1}^{(l)} \big) a \ld_s \big( e_{1, k}^{(l)} \big)
\end{equation}
for $a \in A_s.$
We claim that $E_s (a)$ commutes with $\ld_s (x)$
for $a \in A_s,$ $x \in F,$ and $s = 1, 2.$
It suffices to take $x = e_{i, j}^{(m)}.$
In the product $E_a (a) \ld_s (x),$
the terms coming from~(\ref{Eq:L-SurjComm})
with $l \neq m$ vanish,
leaving
\[
E_s (a) \ld_s (x)
 = \sum_{k = 1}^{r (m)}
  \ld_s \big( e_{k, 1}^{(m)} \big) a
      \ld_s \big( e_{1, k}^{(m)} e_{i, j}^{(m)} \big)
 = \ld_s \big( e_{i, 1}^{(m)} \big) a
      \ld_s \big( e_{1, j}^{(m)} \big).
\]
Similarly, we also get
$\ld_s (x) E_s (a)
 = \ld_s \big( e_{i, 1}^{(m)} \big) a \ld_s \big( e_{1, j}^{(m)} \big),$
proving the claim.

The relation
\[
\sum_{l = 1}^n \sum_{k = 1}^{r (l)} \ld_s \big( e_{k, k}^{(l)} \big)
 = \ld_s (1)
 = 1
\]
implies that for $s = 1, 2$ and $a \in B_s,$
we have $E_s (a) = a.$
It is clear that $E_2 \circ \ph = \ph \circ E_1.$

Now let $b \in B_2.$
Choose $a \in A_1$ such that $\ph (a) = b.$
Then $E_1 (a) \in B_1,$
and $\ph (E_1 (a)) = E_2 (b) = b.$
This completes the proof.
\end{proof}

\begin{thm}\label{T-TensPrdF}
Let $G$ be a compact group,
let $A$ be a \uca,
and let $F$ be a \fd\  \ca.
Let $\af \colon G \to \Aut (A)$
be an equivariantly semiprojective action,
and let $\bt \colon G \to \Aut (F)$ be any action.
Then $\bt \otimes \af \colon G \to \Aut (F \otimes A)$
is equivariantly semiprojective.
\end{thm}

\begin{proof}
Define $\om \colon F \to F \otimes A$ by $\om (x) = x \otimes 1.$
By Theorem~\ref{T:FDEqSj} and Lemma~\ref{L:SjTrtve},
it suffices to prove that
$\om$ is equivariantly conditionally semiprojective.
Let the notation be as in \Def{D:EqsjForHom},
except with $F$ in place of~$A$ and $F \otimes A$ in place of~$B.$
We have the diagram
\[
\xymatrix{
& & C \ar@{-->}[d]^{\kp_n} \ar@/^2pc/[dd]^{\kp} \\
& & C / J_n \ar@{-->}[d]^{\pi_n} \\
F \ar[r]_(0.4){\om} \ar[rruu]^{\ld}
    & F \otimes A \ar[r]_-{\ph} \ar@{-->}[ru]^{\ps} & C / J,
}
\]
in which the solid arrows correspond to given
equivariant unital \hm s.
We must find $n$ and an equivariant unital \hm\  $\ps$
which make the whole diagram commute.

Define
\[
D = \big\{ c \in C \colon
  {\mbox{$c$ commutes with $\ld (x)$ for all $x \in F$}} \big\},
\]
which is a $G$-invariant subalgebra of~$C.$
Define $I_n = J_n \cap D$ for $n \in \Nz,$
and set $I = J \cap D.$
Then $I, I_0, I_1, \ldots$ are $G$-invariant ideals in~$D,$
and $I = {\overline{\bigcup_{n = 0}^{\infty} I_n}}.$
Moreover, Lemma~\ref{L-SurjComm}
implies that
\[
D / I_n = \big\{ c \in C / J_n \colon
  {\mbox{$c$ commutes with $(\kp_n \circ \ld) (x)$
     for all $x \in F$}} \big\}
\]
for $n \in \Nz,$
and that
\begin{equation}\label{Eq:T-TensPrdF-1}
D / I = \big\{ c \in C / J \colon
  {\mbox{$c$ commutes with $(\kp \circ \ld) (x)$
     for all $x \in F$}} \big\}.
\end{equation}

Define an equivariant \hm\  $\ph_0 \colon A \to C / J$
by $\ph_0 (a) = \ph (1 \otimes a)$ for $a \in A.$
By~(\ref{Eq:T-TensPrdF-1}), the range of~$\ph_0$
is contained in $D / I.$
Since $A$ is \eqsj,
there are $n$ and
a unital equivariant \hm\  $\ps_0 \colon A \to D / I_n \subset C / J_n$
such that $\pi_n \circ \ps_0 = \ph_0.$

By construction,
the ranges of $\ps_0$ and $\kp_n \circ \ld$ commute,
so there is a unital \hm\  $\ps_0 \colon A \to C / J_n,$
necessarily equivariant,
such that
\[
\ps (x \otimes a) = (\kp_n \circ \ld) (x) \ps_0 (a)
\]
for all $x \in F$ and $a \in A.$
Using $\pi_n \circ \kp_n \circ \ld = \ph \circ \om,$
we get $\pi_n \circ \ps = \ph.$
\end{proof}

\section{Quasifree actions on Cuntz algebras}\label{Sec:QFCuntz}

\indent
The purpose of this section is to prove that quasifree actions
of compact groups on the Cuntz algebras~$\OA{d}$
and the extended Cuntz algebras~$E_d,$
for $d$ finite,
are \eqsj.
We begin by defining and introducing notation for quasifree actions.

\begin{ntn}\label{N:Cuntz}
Let $d \in \N.$
(We allow $d = 1,$ in which case $\OA{d} = C (S^1).$)
We write $s_1, s_2, \ldots, s_d$ for the standard generators
of the Cuntz algebra~$\OA{d}.$
That is, we take $\OA{d}$ to be generated by
elements $s_1, s_2, \ldots, s_d$
satisfying the
relations $s_j^* s_j = 1$ for $j = 1, 2, \ldots, d$
and $\sum_{j = 1}^d s_j s_j^* = 1.$
\end{ntn}

\begin{ntn}\label{N:ExtCuntz}
Let $d \in \N.$
We recall the extended Cuntz algebra~$E_d.$
It is the universal \uca\  generated by $d$~isometries
with orthogonal range \pj s which are not required to add up to~$1.$
(For $d = 1,$ we get the Toeplitz algebra,
the \ca\  of the unilateral shift,
which here is called~$r_1.$)
We call these isometries $r_1, r_2, \ldots, r_d,$
so that the relations are $r_j^* r_j = 1$ for $j = 1, 2, \ldots, d$
and $r_j r_j^* r_k r_k^* = 0$
for $j, k = 1, 2, \ldots, d$ with $j \neq k.$
When $d$ must be specified,
we write $r_1^{(d)}, r_2^{(d)}, \ldots, r_d^{(d)}.$
We further let $\et \colon E_d \to \OA{d}$ be the quotient map,
defined, following Notation~\ref{N:Cuntz},
by $\et (r_j) = s_j$ for $j = 1, 2, \ldots, d.$
\end{ntn}

\begin{ntn}\label{N-SLambda}
For $\ld = (\ld_1, \ld_2, \ldots, \ld_d) \in \C^d,$
we further define $s_{\ld} \in \OA{d}$ and $r_{\ld} \in E_d$
by
\[
s_{\ld} = \sum_{j = 1}^d \ld_j s_j
\andeqn
r_{\ld} = \sum_{j = 1}^d \ld_j r_j.
\]
\end{ntn}

\begin{ntn}\label{N-MdToOd}
Let $d \in \N.$
We let $( e_{j, k} )_{j, k = 1}^n$ be the standard system
of matrix units in~$M_n.$
We denote by $\mu \colon M_d \to \OA{d}$
the injective unital \hm\  determined by $\mu (e_{j, k}) = s_j s_k^*$
for $j, k = 1, 2, \ldots, d.$
We further denote by $\mu_0 \colon M_d \oplus \C \to E_d$
the injective unital \hm\  determined by
$\mu_0 \big( (e_{j, k}, \, 0) \big) = r_j r_k^*$
for $j, k = 1, 2, \ldots, d$
and $\mu_0 (0, 1) = 1 - \sum_{j = 1}^d r_j r_j^*.$
\end{ntn}

Recall (Notation~\ref{N:UGp}) that $U (A)$ is the unitary group of~$A.$

\begin{ntn}\label{N-AdRep}
Let $A$ be a \uca,
and let $u \in U (A).$
We denote by $\Ad (u)$ the automorphism
$\Ad (u) (a) = u a u^*$ for $a \in A.$
Further let $G$ be a topological group,
and let $\rh \colon G \to U (A)$ be a \ct\  \hm.
We denote by $\Ad (\rh)$ the action $\aGA$
given by $\Ad (\rh)_g = \Ad (\rh (g))$
for $g \in G.$
For $\rh \colon G \to U (M_d),$
we also write $\Ad (\rh \oplus 1)$ for the action
$g \mapsto \Ad (\rh (g), \, 1)$ on $M_n \oplus \C.$
We always take $M_d \oplus \C$ to have this action.
\end{ntn}

We now give the basic properties of quasifree actions on~$E_d.$

\begin{lem}\label{L:ExtQFree}
Let $G$ be a topological group,
let $d \in \N,$
and let $\rh \colon G \to U (M_d)$
be a unitary representation of~$G$ on~$\C^d.$
Then there exists a unique action $\af^{\rh} \colon G \to \Aut (E_d)$
such that,
with $\mu_0$ as in Notation~\ref{N-MdToOd},
for $j = 1, 2, \ldots, d$ and $g \in G$ we have
\[
\af^{\rh}_g (r_j) = \mu_0 (\rh (g), \, 1) r_j.
\]
Moreover, this action has the following properties:
\begin{enumerate}
\item\label{L:ExtQFree-0}
For all $g \in G,$ if we write
\[
\rh (g)
 = \sum_{j, k = 1}^d \rh_{j, k} (g) e_{j, k}
 = \left( \begin{matrix}
\rh_{1, 1} (g) & \rh_{1, 2} (g) & \cdots & \rh_{1, d} (g)  \\
\rh_{2, 1} (g) & \rh_{2, 2} (g) & \cdots & \rh_{2, d} (g)  \\
\vdots         & \vdots         & \ddots &  \vdots         \\
\rh_{d, 1} (g) & \rh_{d, 2} (g) & \cdots & \rh_{d, d} (g)
\end{matrix} \right),
\]
then
\[
\af^{\rh}_g (r_k) = \sum_{j = 1}^d \rh_{j, k} (g) r_j
\]
for $k = 1, 2, \ldots, d.$
\item\label{L:ExtQFree-1}
Following Notation~\ref{N-SLambda},
for every $\ld \in \C^d$ and $g \in G,$
we have $\af^{\rh}_g (r_{\ld}) = r_{\rh (g) \ld}.$
\item\label{L:ExtQFree-2}
The \hm\  $\mu_0$ is equivariant.
(Recall the action on $M_d \oplus \C$ from Notation~\ref{N-AdRep}.)
\item\label{L:ExtQFree-3}
The \pj\  $1 - \sum_{j = 1}^d r_j r_j^* \in E_d$
is $G$-invariant.
\end{enumerate}
\end{lem}

\begin{proof}
For $g \in G,$
we claim that the elements
$\mu_0 (\rh (g), \, 1) r_j$ satisfy the relations defining~$E_d.$
Because $\mu_0 (\rh (g), \, 1)$ is unitary,
we have
\[
\big( \mu_0 (\rh (g), \, 1) r_j \big)^*
  \big( \mu_0 (\rh (g), \, 1) r_k \big) = r_j^* r_k
\]
for $j, k = 1, 2, \ldots, d.$
Using the definition of~$\mu_0$ at the second step,
we also get
\begin{align}\label{Eq:FixDefect}
\lefteqn{\sum_{j = 1}^d \big[ \mu_0 (\rh (g), \, 1) r_j \big]
          \big[ \mu_0 (\rh (g), \, 1) r_j \big]^*}
          \\
& \hspace*{2em} {\mbox{}}
  = \mu_0 (\rh (g), \, 1) \left( \sssum{j = 1}{d} r_j r_j^* \right)
           \mu_0 (\rh (g), \, 1)
  = \sum_{j = 1}^d r_j r_j^*.
\notag
\end{align}
It is now easy to prove the claim.

It follows that there is a unique \hm\  %
$\af^{\rh}_g \colon E_d \to E_d$
such that $\af^{\rh}_g (r_j) = \mu_0 (\rh (g), \, 1) r_j$
for $j = 1, 2, \ldots, d.$
Part~(\ref{L:ExtQFree-3}) follows from~(\ref{Eq:FixDefect}).
Part~(\ref{L:ExtQFree-0}) is just a calculation,
and implies part~(\ref{L:ExtQFree-1})
when $\ld \in \C^d$ is a standard basis vector.
The general case of part~(\ref{L:ExtQFree-1})
follows by linearity.

Part~(\ref{L:ExtQFree-1}) implies that
$\af^{\rh}_{g h} (s_{\ld})
   = \af^{\rh}_{g} \circ \af^{\rh}_{h} (s_{\ld})$
for $g, h \in G$ and $\ld \in \C,$
and also $\af^{\rh}_1 = \id_{E_d},$
so $g \mapsto \af^{\rh}_g$ is a \hm\  to $\Aut (E_d).$
Continuity of $g \mapsto \af_g (a)$
for $a \in E_d$ follows from the fact that it holds
whenever $a = r_j$ for some~$j.$

It remains to prove~(\ref{L:ExtQFree-2}).
For $j, k = 1, 2, \ldots, d,$
we have
\begin{align*}
\af^{\rh}_g \big( \mu_0 (e_{j, k}, \, 0 ) \big)
& = \af^{\rh}_g (r_j) \af^{\rh}_g (r_k^*)
  = \mu_0 (\rh (g), \, 1) r_j r_k^* \mu_0 (\rh (g), \, 1)^*
        \\
& = \mu_0 \big( (\rh (g), \, 1) (e_{j, k}, \, 0 )
                  (\rh (g), \, 1)^* \big)
  = \mu_0 \big( \Ad (\rh \oplus 1)_g (e_{j, k}, \, 0 ) \big),
\end{align*}
as desired.
We must also to check the analogous equation with $(0, 1)$ in place
of $(e_{j, k}, 0 ),$
but this is immediate from part~(\ref{L:ExtQFree-3}).
\end{proof}

Here are the corresponding properties for quasifree actions
on~$\OA{d}.$
These are mostly well known, and are stated for reference
and to establish notation.
They also follow from \Lem{L:ExtQFree}.

\begin{lem}\label{L:OdQFree}
Let $G$ be a topological group,
let $d \in \N,$
and let $\rh \colon G \to U (M_d)$
be a unitary representation of~$G$ on~$\C^d.$
Then there exists a unique action $\bt^{\rh} \colon G \to \Aut (\OA{d})$
such that,
with $\mu$ as in Notation~\ref{N-MdToOd},
for $j = 1, 2, \ldots, d$ and $g \in G$ we have
\[
\bt^{\rh}_g (s_j) = \mu (\rh (g), \, 1) s_j.
\]
Moreover, this action has the following properties:
\begin{enumerate}
\item\label{L:OdQFree-0}
For all $g \in G,$ if we write
\[
\rh (g) = \sum_{j, k = 1}^d \rh_{j, k} (g) e_{j, k},
\]
then
\[
\bt^{\rh}_g (s_k) = \sum_{j = 1}^d \rh_{j, k} (g) s_j
\]
for $k = 1, 2, \ldots, d.$
\item\label{L:OdQFree-1}
Following Notation~\ref{N-SLambda},
for every $\ld \in \C^d$ and $g \in G,$
we have $\bt^{\rh}_g (s_{\ld}) = s_{\rh (g) \ld}.$
\item\label{L:OdQFree-2}
When $M_d$ is equipped with the action
$\Ad (\rh),$
the \hm\  $\mu$ is equivariant.
\item\label{L:OdQFree-3}
The quotient map $\et$ from $(G, E_d, \af^{\rh})$
to $(G, \OA{d}, \bt^{\rh})$ is equivariant.
\end{enumerate}
\end{lem}

\begin{proof}
\Lem{L:ExtQFree}(\ref{L:ExtQFree-3})
implies that the ideal in $E_d$ generated by
$\sum_{j = 1}^d r_j r_j^*$ is invariant.
Therefore the quotient is a \ga.
It is well known that we may identify
$\et \colon E_d \to \OA{d}$ with this quotient map.
So we have an action $\bt^{\rh} \colon G \to \Aut (\OA{d}).$
It is clear from the construction and the fact
that $\et (\mu_0 (e_{j, k}, \, 0)) = \mu (e_{j, k})$
for $j, k = 1, 2, \ldots, d$
that
this action satisfies
$\bt^{\rh}_g (s_j) = \mu (\rh (g), \, 1) s_j$
for $j = 1, 2, \ldots, d$ and $g \in G.$
Uniqueness of $\bt^{\rh}$ is clear.
Similarly, parts (\ref{L:OdQFree-0}), (\ref{L:OdQFree-1}),
and~(\ref{L:OdQFree-2})
follow from the corresponding formulas in \Lem{L:ExtQFree}.
\end{proof}

The algebraic computations we need for equivariant semiprojectivity
of quasifree actions are contained in the following lemma.

\begin{lem}\label{L-MakeCcyEd}
Let $G$ be a topological group.
Let $d \in \N,$
let $\rh \colon G \to U (M_d)$
be a unitary representation of~$G$ on~$\C^d,$
and let $\af^{\rh} \colon G \to \Aut (E_d)$ be the quasifree action
of \Lem{L:ExtQFree}.
Let $\mu_0 \colon M_d \oplus \C \to E_d$
be as in Notation~\ref{N-MdToOd},
and recall (Notation~\ref{N-AdRep})
the action $\Ad (\rh \oplus 1)$ on $M_d \oplus \C.$
Let $(G, C, \gm)$ be a unital \ga,
and let $\ph \colon E_d \to C$
be a unital \hm\  such that $\ph \circ \mu_0$ is \eqv.
For $g \in G,$ define
\[
w (g) = \ph ( \af^{\rh}_g (r_1))^* \gm_g (\ph (r_1)).
\]
Then:
\begin{enumerate}
\item\label{L-MakeCcyEd-1}
$g \mapsto w (g)$ is a \cfn\  from $G$ to $U (C).$
\item\label{L-MakeCcyEd-2}
For $j = 1, 2, \ldots, d$ and $g \in G,$
we have $\ph ( \af^{\rh}_g (r_{j})) w (g) = \gm_g (\ph (r_{j})).$
\item\label{L-MakeCcyEd-3}
For every $g, h \in G,$
we have $w (g h) = w (g) \gm_g (w (h)).$
\item\label{L-MakeCcyEd-4}
For every $g \in G,$
we have
$\| w (g) - 1 \|
   = \big\| \ph ( \af^{\rh}_g (r_1)) - \gm_g (\ph (r_1)) \big\|.$
\item\label{L-MakeCcyEd-5}
If $v \in U (C)$ satisfies $v \gm_g (v)^* = w (g)$
for all $g \in G,$
then there is a unique unital \hm\  $\ps \colon E_d \to C$
such that $\ps (r_j) = \ph (r_j) v$ for $j = 1, 2, \ldots, d.$
Moreover, $\ps$ is \eqv\  and $\ps \circ \mu_0 = \ph \circ \mu_0.$
\item\label{L-MakeCcyEd-6}
If $\kp \colon C \to D$ is an \eqv\  \hm\  %
from $C$ to some other \ga~$D,$
and $\kp \circ \ph$ is \eqv,
then $\kp (w (g)) = 1$ for all $g \in G.$
\end{enumerate}
\end{lem}

\begin{proof}
We use the usual notation for matrix units,
as in Notation~\ref{N-MdToOd}.
We also recall (\Lem{L:ExtQFree}(\ref{L:ExtQFree-2}))
that $\mu_0$ is \eqv.

We prove~(\ref{L-MakeCcyEd-1})
by showing that $\ph ( \af^{\rh}_g (r_1))$ and $\gm_g (\ph (r_1))$
are isometries with the same range \pj.
It is clear that both are isometries.
The range \pj s are
\[
\ph ( \af^{\rh}_g (r_1)) \ph ( \af^{\rh}_g (r_1))^*
  = \ph ( \af^{\rh}_g (r_1 r_1^*) )
  = (\ph \circ \mu_0)
     \big( \Ad (\rh \oplus 1)_g ( e_{1, 1}, \, 0 ) \big)
\]
and
\[
\gm_g (\ph (r_1)) \gm_g (\ph (r_1))^*
  = \gm_g (\ph (r_1 r_1^*))
  = \gm_g \big( (\ph \circ \mu_0) (e_{1, 1}, \, 0) \big).
\]
These are equal because $\ph \circ \mu_0$ is \eqv.
So (\ref{L-MakeCcyEd-1}) follows.

For~(\ref{L-MakeCcyEd-2}),
we have,
using equivariance of both~$\mu_0$
and $\ph \circ \mu_0$
at the third step,
\begin{align*}
\ph ( \af^{\rh}_g (r_j)) w_g
  & = \ph ( \af^{\rh}_g (r_j))
        \ph ( \af^{\rh}_g (r_1))^* \gm_g (\ph (r_1))
    = (\ph \circ \af^{\rh}_g \circ \mu_0) (e_{j, 1}, \, 0)
       (\gm_g \circ \ph) (r_1)
     \\
  & = (\gm_g \circ \ph \circ \mu_0) (e_{j, 1}, \, 0)
       (\gm_g \circ \ph) (r_1)
    = (\gm_g \circ \ph) (r_j r_1^* r_1)
    = (\gm_g \circ \ph) (r_j),
\end{align*}
as desired.

For~(\ref{L-MakeCcyEd-3}),
we simplify the notation by defining
\begin{equation}\label{Eq:ugDef}
u (g) = (\ph \circ \mu_0) ( \rh (g), 1)
\end{equation}
for $g \in G.$
Then $u (g)$ is unitary.
By the definition of $\af^{\rh},$
we have
\begin{equation}\label{Eq:ugphi}
\ph ( \af^{\rh}_g (r_{j}))
  = u (g) \ph (r_{j})
\end{equation}
for $g \in G$ and $j = 1, 2, \ldots, d.$
By equivariance of $\ph \circ \mu_0,$
we have
\begin{equation}\label{Eq:ugmu0}
(\gm_g \circ \ph \circ \mu_0) (x)
 = \big( \Ad (u (g)) \circ \ph \circ \mu_0 \big) (x)
\end{equation}
for $g \in G$
and $x \in M_d \oplus \C.$
Using~(\ref{Eq:ugphi}) at the first step,
(\ref{Eq:ugDef})~and (\ref{Eq:ugmu0}) at the second step,
$\ph (r_1 r_1^*) = \mu_0 (e_{1, 1}, \, 0)$ and~(\ref{Eq:ugmu0})
at the third step,
and $r_1 r_1^* r_1 = r_1,$
(\ref{Eq:ugphi}), and $u (g) u (h) = u (g h)$ at the last step,
for $g, h \in G$ we get
\begin{align*}
w (g) \gm_g (w (h))
& = \big[ \ph (r_1)^* u (g)^* \gm_g (\ph (r_1)) \big]
      \gm_g \big( \ph (r_1)^* u (h)^* \gm_h (\ph (r_1)) \big)
    \\
& = \ph (r_1)^* u (g)^* \gm_g (\ph (r_1 r_1^*))
      \big[ u (g) u (h)^* u (g)^* \big] \gm_{g h} (\ph (r_1))
    \\
& = \ph (r_1)^* \ph (r_1 r_1^*)
      u (h)^* u (g)^* \gm_{g h} (\ph (r_1))
  = w (g h).
\end{align*}
This proves~(\ref{L-MakeCcyEd-3}).

For~(\ref{L-MakeCcyEd-4}),
use the fact that
$\ph ( \af^{\rh}_g (r_1))$ is an isometry at the first step
and part~(\ref{L-MakeCcyEd-2}) at the second step
to write
\[
\| w (g) - 1 \|
   = \big\| \ph ( \af^{\rh}_g (r_1)) w (g)
         - \ph ( \af^{\rh}_g (r_1)) \big\|
  = \big\| \gm_g (\ph (r_1)) - \ph ( \af^{\rh}_g (r_1)) \big\|.
\]

We prove~(\ref{L-MakeCcyEd-5}).
Existence and uniqueness of $\ps$ are true for any unitary~$v,$
because the elements $\ph (r_j) v$ are isometries with
orthogonal ranges.
For $j, k = 1, 2, \ldots, d,$ we have
\[
(\ps \circ \mu_0) (e_{j, k}, \, 0)
  = \ps (r_j) \ps (r_k^*)
  = \big( \ph (r_j) v \big) \big( v^* \ps (r_k^*) \big)
  = (\ph \circ \mu_0) (e_{j, k}, \, 0).
\]
Since also $(\ps \circ \mu_0) (1) = (\ph \circ \mu_0) (1),$
it follows that $\ps \circ \mu_0 = \ph \circ \mu_0.$

It remains to prove that $\ps$ is \eqv.
In the following calculation,
we let $u (g)$ be as in~(\ref{Eq:ugDef}).
We use (\ref{L-MakeCcyEd-2}) and~(\ref{Eq:ugphi}) at step~3,
and (\ref{Eq:ugDef})~and
$\ps \circ \mu_0 = \ph \circ \mu_0$ at step~5,
to get,
for $g \in G$ and $j = 1, 2, \ldots, d,$
\begin{align*}
\gm_g (\ps (r_j))
& = \gm_g (\ph (r_j)) \gm_g (v)
  = \gm_g (\ph (r_j)) w (g)^* v
   \\
& = u (g) \ph (r_j) v
  = u (g) \ps (r_j)
  = \ps ( \mu_0 ( \rh (g), \, 1) r_j)
  = \ps ( \af^{\rh}_g ( r_j)).
\end{align*}
Equivariance of~$\ps$ follows.

Part~(\ref{L-MakeCcyEd-6}) is immediate.
\end{proof}

\Lem{L-MakeCcyEd}
will be used to produce cocycles which are close to~$1.$
To deal with them,
we need results similar to \Lem{L:OneStep}
and \Lem{L:A7}.

\begin{lem}\label{L:1StepCcy}
Let $G$ be a compact group with normalized Haar measure~$\mu.$
Let $\GAa$ be a unital \ga,
and let $w \colon G \to U (A)$
be a \cfn\  such that for all $g, h \in G$ we have
$w (g h) = w (g) \af_g (w (h)).$
Suppose $r \in \big[ 0, \tfrac{1}{5} \big],$
and let $v \in U (A)$
satisfy
\[
\| v \af_g (v)^* - w (g) \| \leq r
\]
for all $g \in G.$
Define
\[
z = v \exp \left( \int_{G}
    \log \Big( v^* \af_h^{-1} \big( w (h)^* v \big) \Big) \, d \mu (h)
              \right).
\]
Then $z \in U (A)$
and satisfies
\[
\| z \af_g (z)^* - w (g) \| \leq 10 r^2
\]
for all $g \in G$
and
\[
\| z - v \| \leq 2 r.
\]
\end{lem}

\begin{proof}
For every $h \in G,$
we have
\begin{equation}\label{Eq:1StOrigEst}
\big\| v^* \af_h^{-1} \big( w (h)^* v \big) - 1 \big\|
  = \big\| \big[ w (h) - v \af_h (v)^* \big]^* \big\|
  \leq r.
\end{equation}
Since $r < 1,$ the logarithm in the formula for $v$ exists,
so $v$ is well defined.
Moreover,
\[
\log \Big( v^* \af_h^{-1} \big( w (h)^* v \big) \Big)
   \in i A_{\sa}
\]
for $h \in G,$
so $v \in U (A)$
implies $z \in U (A).$

Define
\[
z_0 = \int_{G} \af_h^{-1} \big( w (h)^* v \big) \, d \mu (h).
\]
Using $\| v \af_h (v)^* - w (h) \| \leq r$
at the third step,
we get
\[
\| z_0 - v \|
  \leq \int_{G}
      \big\| \af_h^{-1} \big( w (h)^* v \big) - v \big\| \, d \mu (h)
  = \int_{G} \big\| w (h) - \af_h (v) v^* \big\| \, d \mu (h)
  \leq r.
\]
Then,
making the change of variables $h$ to $h g$ at the first step
and
using $w (h g) = w (h) \af_h (w (g))$ at the second step,
for $g \in G$ we have
\[
\af_g (z_0)
  = \int_{G} \af_{h}^{-1} \big( w (h)^* v \big) \, d \mu (h)
  = \int_{G} w (g)^* \af_{h}^{-1} \big( w (h)^* v \big) \, d \mu (h)
  = w (g)^* z_0.
\]
Rewriting
\[
z_0 = v \int_{G} v^* \af_h^{-1} \big( w (h)^* v \big) \, d \mu (h)
\]
and using~(\ref{Eq:1StOrigEst}),
we can apply \Lem{L-IntegralEst}
to get
\[
\| z - z_0 \| \leq \frac{5 r^2}{2 (1 - 2 r)}
              \leq 5 r^2.
\]
It follows from the equation $\af_g (z_0) = w (g)^* z_0$
that $\| \af_g (z) - w (g)^* z \| \leq 10 r^2.$
The first estimate in the conclusion follows.

Since $r \leq \frac{1}{5}$ and we already know $\| z_0 - v \| \leq r,$
we also get $\| z - v \| \leq 2 r.$
This is the second estimate in the conclusion.
\end{proof}

\begin{lem}\label{L-A7ForCocycle}
Let $G$ be a compact group with normalized Haar measure~$\mu.$
Let $\GAa$ and $\GBb$ be unital \ga s,
and let $\kp \colon A \to B$ be a unital \eqv\ \hm.
Let $w \colon G \to U (A)$
be a \cfn\  such that for all $g, h \in G$ we have
$w (g h) = w (g) \af_g (w (h)).$
Suppose $0 \leq r < \tfrac{1}{10},$
and let $v_0 \in U (A)$
satisfy
\[
\| v_0 \af_g (v_0)^* - w (g) \| \leq r
\andeqn
\kp (v_0) \bt_g ( \kp (v_0) )^* = \kp (w (g))
\]
for all $g \in G.$
Inductively define $v_m \in U (A)$
by (following \Lem{L:1StepCcy})
\[
v_{m + 1} = v_m \exp \left( \int_{G}
 \log \Big( v_m^* \af_h^{-1} \big( w (h)^* v_m \big) \Big) \, d \mu (h)
              \right).
\]
Then for every $m \in \N$
the element $v_m$ is a well defined unitary in~$A$
such that $\kp (v_m) = \kp (v).$
Moreover, $v = \limi{m} v_m$ exists and satisfies
$v \af_g (v)^* = w (g)$ for all $g \in G,$
and also
\[
\| v - v_0 \| \leq \frac{2 r}{1 - 10 r}
\andeqn
\kp (v) = \kp (v_0).
\]
\end{lem}

\begin{proof}
The proof is essentially the same
as the proof of \Lem{L:A7}.
One proves by induction that for all $m \in \Nz,$
the element $v_m$ is well defined, in $U (A),$
and satisfies
\[
\kp (v_m) = \kp (v)
\andeqn
\| v_m - v_{m - 1} \| \leq 2 r (10 r)^{m - 1},
\]
and that for $g \in G$ we have
\[
\| v_m \af_g (v_m)^* - w (g) \| \leq r (10 r)^m.
\]
We omit further details.
\end{proof}

\begin{thm}\label{T-QFEqSj}
Let $G$ be a compact group,
let $d \in \N,$
and let $\rh \colon G \to U (M_d)$
be a unitary representation of~$G$ on~$\C^d.$
Then the quasifree action
$\af^{\rh} \colon G \to \Aut (E_d)$ of \Lem{L:ExtQFree}
is \eqsj.
\end{thm}

\begin{proof}
Let $\mu_0 \colon M_d \oplus \C \to E_d$
be as in Notation~\ref{N-MdToOd}.
Recall (Notation~\ref{N-AdRep})
the action $\Ad (\rh \oplus 1)$ on $M_d \oplus \C.$
Then $\mu_0$ is \eqv\  %
by \Lem{L:ExtQFree}(\ref{L:ExtQFree-2}).
The action $\Ad (\rh \oplus 1)$ is \eqsj\  by Theorem~\ref{T:FDEqSj}.
By \Lem{L:SjTrtve}, it therefore it suffices to prove
that $\mu_0$ is equivariantly conditionally semiprojective
in the sense of \Def{D:EqsjForHom}.

We adopt the notation of
\Def{D:EqsjForHom} and Remark~\ref{R:EqSj}.
Thus,
assume that $\ld \colon M_d \oplus \C \to C$
and $\ph \colon B \to C / J$
are unital equivariant \hm s
such that $\kp \circ \ld = \ph \circ \om.$
Since $E_d$ is \sj\  without the group,
there exists $n_0 \in \N$ and a unital \hm\  %
$\nu_0 \colon E_d \to C / J_{n_0}$ such that
$\pi_{n_0} \circ \nu_0 = \ph.$
In particular,
\[
\pi_{n_0} \circ \nu_0 \circ \mu_0
  = \pi_{n_0} \circ \kp_{n_0} \circ \ld.
\]

For $k \geq n_0$ define $f_k \colon U (M_d \oplus \C) \to [0, \infty)$
by
\[
f_k (x) = \big\| \pi_{k, n_0} \big( ( \nu_0 \circ \mu_0 ) (x)
  - ( \kp_{n_0} \circ \ld ) (x) \big) \big\|.
\]
for $x \in U (M_d \oplus \C).$
The functions $f_n$ are \ct, and satisfy
\[
f_{n_0} \geq f_{n_0 + 1} \geq f_{n_0 + 2} \geq \cdots
\]
and $f_k \to 0$ pointwise.
Since $U (M_d \oplus \C)$ is compact,
Dini's Theorem (Proposition~11 in Chapter~9 of~\cite{Ry})
implies that $f_k \to 0$ uniformly.
Therefore there exists $n_1 \geq n_0$ such that
for all $x \in U (A),$
we have
\[
\big\| \pi_{n_1, n_0} \big( ( \nu_0 \circ \mu_0 ) (x)
  - ( \kp_{n_0} \circ \ld ) (x) \big) \big\|
  < \tfrac{1}{2}.
\]
Proposition~\ref{P:CloseImpUE}
provides a unitary $u \in U (C / J_{n_1})$
such that $\pi_{n_1} (u) = 1$ and
\[
u \big( \pi_{n_1, n_0} \circ \nu_0 \circ \mu_0 \big) (x) u^*
  = \big( \pi_{n_1, n_0} \circ \kp_{n_0} \circ \ld \big) (x)
\]
for all $x \in U (M_d \oplus \C).$
Define $\nu_1 \colon E_d \to C / J_{n_1}$
by $\nu_1 = \Ad (u) \circ \pi_{n_1, n_0} \circ \nu_0.$
Then $\pi_{n_1} \circ \nu_1 = \ph$
and $\kp_{n_1} \circ \ld = \nu_1 \circ \mu_0.$

For $k \geq n_1$, the functions
\[
g \mapsto \big\| \pi_{k, n_1} \big( ( \gm^{(k)}_g \circ \nu_1 ) (r_1)
  - ( \nu_1 \circ \af^{\rh}_g ) (r_1) \big) \big\|
\]
are \ct\  and pointwise nonincreasing as $k \to \infty.$
Since $\pi_{n_1}$ and $\pi_{n_1} \circ \nu_1 = \ph$
are \eqv,
these functions converge pointwise to zero.
Another application of Dini's Theorem provides
$n \geq n_1$ such that,
with $\nu = \pi_{n, n_1} \circ \nu_1$
and using equivariance of $\pi_{n, n_1},$
we have
\[
\sup_{g \in G}
 \big\| ( \gm^{(n)}_g \circ \nu ) (r_1)
      - ( \nu \circ \af^{\rh}_g ) (r_1) \big\|
  < \frac{1}{20}.
\]

Now let $w (g)$ be as in \Lem{L-MakeCcyEd},
with $C / J_n$ in place of~$C$ and $\nu$ in place of~$\ph.$
Then $\sup_{g \in G} \| w (g) - 1 \| < \frac{1}{20}$
by \Lem{L-MakeCcyEd}(\ref{L-MakeCcyEd-4})
and $\pi_n (w (g)) = 1$ for all $g \in G$
by \Lem{L-MakeCcyEd}(\ref{L-MakeCcyEd-6}).
Using these facts,
\Lem{L-MakeCcyEd}(\ref{L-MakeCcyEd-1}),
and the cocycle condition of
\Lem{L-MakeCcyEd}(\ref{L-MakeCcyEd-3}),
we can apply
\Lem{L-A7ForCocycle}
with $v_0 = 1$
to find $v \in U (C / J_n)$
such that $\pi_n (v) = 1$
and $v \af_g (v)^* = w (g)$ for all $g \in G.$
Let $\ps \colon E_d \to C / J_n$ be as in
\Lem{L-MakeCcyEd}(\ref{L-MakeCcyEd-5})
with this choice of~$v.$
Then $\ps$ is \eqv\  and $\ps \circ \mu_0 = \nu \circ \mu_0$
by \Lem{L-MakeCcyEd}(\ref{L-MakeCcyEd-5}).
Since $\nu_1 \circ \mu_0 = \kp_{n_1} \circ \ld,$
we get $\ps \circ \mu_0 =  \kp_n \circ \ld.$
Since $\pi_n (v) = 1,$
we get $\pi_n (\ps (r_j)) = \pi_n (\nu (r_j)) = \ph (r_j)$
for $j = 1, 2, \ldots, d.$
Therefore
$\pi_n \circ \ps = \ph.$
This completes the proof that $\mu_0$
is equivariantly conditionally semiprojective.
\end{proof}

\begin{cor}\label{C-OdQFEqSj}
Let $G$ be a compact group,
let $d \in \N,$
and let $\rh \colon G \to U (M_d)$
be a unitary representation of~$G$ on~$\C^d.$
Then the quasifree action
$\bt^{\rh} \colon G \to \Aut (\OA{d})$ of \Lem{L:OdQFree}
is \eqsj.
\end{cor}

\begin{proof}
By Theorem~\ref{T-QFEqSj} and \Lem{L:SjTrtve},
it suffices to prove
that the quotient map $\et \colon E_d \to \OA{d}$
is equivariantly conditionally semiprojective
in the sense of \Def{D:EqsjForHom}.

Let the notation be as in
\Def{D:EqsjForHom},
except that the map called $\om$ there is~$\et.$
Set $f = \ld \left( 1 - \sum_{j = 1}^d r_j r_j^* \right),$
which is a \pj\  in~$C.$
Then
\[
\kp (f) = \ph \left( 1 - \sssum{j = 1}{d} s_j s_j^* \right) = 0.
\]
Therefore there is $n \in \N$ such that $\| \kp_n (f) \| < 1.$
Since $\kp_n (f)$ is a \pj,
this means that $\kp_n (f) = 0,$
that is,
$(\kp_n \circ \ld) \left( 1 - \sum_{j = 1}^d r_j r_j^* \right) = 0.$
Therefore there is $\ps \colon \OA{d} \to C / J_n$
such that $\kp_n \circ \ld = \ps \circ \et.$
Since $\et$ is surjective,
equivariance of~$\ps$
follows from equivariance of $\et$ and $\kp_n \circ \ld.$
Similarly,
from $\pi_n \circ \ps \circ \et = \kp \circ \ld = \ph \circ \et$
we get $\pi_n \circ \ps = \ph.$
\end{proof}

\begin{rmk}\label{R-RegRep}
An important example of a quasifree action is
the one coming from the regular representation of a finite group.
In this case, one can prove equivariant semiprojectivity without
using any of the machinery developed in this section.

Let $d = \card (G).$
We discuss only~$E_d,$
but the result for $\OA{d}$ can be treated the same way,
or reduced to the result for~$E_d$ as in Corollary~\ref{C-OdQFEqSj}.
Label the generators of $E_d$ as $r_g$ for $g \in G,$
and consider the action $\af \colon G \to \Aut (E_d)$
determined by $\af_g (r_h) = r_{g h}$ for $g, h \in G.$
Following the beginning of the proof of
Theorem~\ref{T-QFEqSj},
we reduce to the situation in which we have
a nonequivariant lifting $\nu \colon E_d \to C / J_n$
of~$\ph$
and an isometry $v_1 \in C / J_n$
such that $v_1 v_1^* = (\kp_n \circ \mu_0) (e_{1, 1}, \, 0)$
and $\| \pi_n (v_1) - \ph (r_1) \|$ is small.
Functional calculus arguments can be used to
replace $v_1$ by a nearby partial isometry $w_1$
such that $w_1 w_1^* = (\kp_n \circ \mu_0) (e_{1, 1}, \, 0)$
and $\pi_n (w_1) = \ph (r_1).$
Then there is a unique \hm\  $\ps \colon E_d \to C / J_n$
such that $\ps (r_g) = \gm^{(n)}_g (r_1),$
and this \hm\  is easily seen to be an \eqv\  lifting of~$\ph$
which satisfies $\kp_n \circ \ld = \ps \circ \mu_0.$
\end{rmk}

\begin{rmk}\label{R-CS1}
We describe what happens when $d = 1.$
In this case, $\OA{d}$ becomes $C (S^1)$
and $E_d$ becomes the \ca\  $C^* (s)$ of the unilateral shift~$s.$
Quasifree actions are those that factor through
the action of $S^1$ on $C (S^1)$ coming from the translation action
of $S^1$ on $S^1,$
and those that factor through
the action of $S^1$ on $C^* (s)$ coming from the automorphisms
determined by $\bt_{\zt} (s) = \zt s$ for $\zt \in S^1.$

Thus, for example,
we conclude that the translation action of $S^1$
on $C (S^1)$ is \eqsj.
This, however, is easy to prove directly.
A unital \eqv\  \hm\  from $C (S^1)$ with translation to
$C / J$ with the action $\gm^{(\infty)}$ is
just a unitary $u \in C / J$
such that $\gm^{(\infty)}_{\zt} (u) = \zt u$
for all $\zt \in S^1.$
This unitary can be partially lifted to a unitary $v$ in
some $C / J_n$
such that $\big\| \gm^{(n)}_{\zt} (v) - \zt v \big\|$ is small
for all $\zt \in S^1.$
To get an exactly equivariant lift~$w,$
set
\[
a = \int_{S^1} \zt^{-1} \gm^{(n)}_{\zt} (v) \, d \zt
\]
(using normalized Haar measure on~$S^1$),
and take $w = a (a^* a)^{- 1/2}.$
\end{rmk}

\section{Quasifree actions on~${\mathcal{O}}_{\infty}$}\label{Sec:QFOI}

\indent
The purpose of this section is to prove that quasifree actions
of finite groups on the Cuntz algebras~$\OI$
are \eqsj.
We begin with a discussion of quasifree actions on~$\OI.$
We will need to include a point of view
different from that of \Lem{L:ExtQFree} and \Lem{L:OdQFree},
primarily to take advantage
of the KK-theory computations in~\cite{Pm}.

\begin{ntn}\label{N-ForOI}
For $d \in \N,$ we make $\C^d$ into a Hilbert space
in the standard way.
For $d = \infty,$ we take $\C^d = l^2 (\N).$
We let $\dt_j \in \C^d$ (for $j = 1, 2, \ldots, d$ or for $d \in \N$)
denote the $j$th standard basis vector,
and we let $U_d$ be the group of unitary operators on~$\C^d,$
with the strong operator topology.
For convenience, we also define $E_{\infty} = \OI,$
and denote its generators by $r_j^{(\infty)}$ as well as by~$s_j.$
\end{ntn}

Of course, when $d$ is finite, the topology on~$U_d$
is the same as the norm topology.
We warn that the notation $\C^{\infty}$ conflicts with notation
often used for the product or the algebraic direct sum
(we are using the Hilbert direct sum),
and that $U_{\infty}$ conflicts with notation sometimes used
for the (much smaller) algebraic direct limit of the groups~$U_d.$

We summarize various results from~\cite{Pm},
and relate them to the viewpoint
of \Lem{L:ExtQFree}
and \Lem{L:OdQFree}.
For a \ca~$A,$
a Hilbert $A$--$A$ bimodule $F$ is as described
at the beginning of Section~1 of~\cite{Pm}:
$F$~is a right Hilbert $A$-module,
with $A$-valued scalar product which is conjugate linear
in the first variable,
together with an injective \hm\  $\ph \colon A \to L (F).$

\begin{thm}[Pimsner]\label{T-PmOne}
Let $d \in \N \cup \{ \infty \},$
and follow Notation~\ref{N-ForOI}.
Make $\C^d$ into a Hilbert $\C$--$\C$ bimodule~$F_d,$
as described above in the obvious way,
with $\ph ( \ld) = \ld \cdot 1_{L (F_d)}$ for $\ld \in \C.$
Let ${\mathcal{T}}_d$ be the associated Toeplitz algebra
${\mathcal{T}}_{F_d}$ as described
in Definition~1.1 of~\cite{Pm},
and call its generators $T_{\xi}$ as there.
Then:
\begin{enumerate}
\item\label{T-PmOne-1}
There is a unique \ct\  action
$\gm^{(d)} \colon U_d \to \Aut ( {\mathcal{T}}_d )$
which satisfies $\gm^{(d)}_u (T_{\xi}) = T_{u \xi}$
for all $\xi \in \C^d.$
\item\label{T-PmOne-2}
There is a unique isomorphism $\sm_d \colon E_d \to {\mathcal{T}}_d$
such that $\sm_d \big( r_j^{(d)} \big) = T_{\dt_j}$ for all~$j.$
(Recall that for $d = \infty$ this means
$\sm_{\infty} \colon \OI \to {\mathcal{T}}_{\infty}.$)
\item\label{T-PmOne-3}
If $d < \infty,$
then $\sm_d$ is equivariant for the action of $U_d$ on $E_d$
gotten by taking $\rh = \id_{U_d}$ in \Lem{L:ExtQFree}
and the action of part~(\ref{T-PmOne-1}).
\item\label{T-PmOne-4}
If $d = \infty,$
when $u \in U_{\infty}$ is written as a matrix
$u = (u_{j, k})_{j, k = 1}^{\infty},$
we have
\[
\big( \sm_d^{-1} \circ \gm^{(d)}_u \circ \sm_d \big)
         \big( r_k^{(d)} \big)
    = \sum_{j = 1}^d u_{j, k} r_j^{(d)}
\]
for all~$k \in \N,$
with convergence in the norm topology on the right.
\item\label{T-PmOne-5}
Let $d_1 \in \N$ and $d_2 \in \N \cup \{ \infty \}$
satisfy $d_1 \leq d_2,$
and set $G = U_{d_1} \times U_{d_2 - d_1}.$
Let $G$ act on $E_{d_1}$ by projection to the first factor
followed by the action on $E_{d_1}$ corresponding to~$\gm^{(d)},$
and let $G$ act on $E_{d_2}$ by the inclusion of $G$ in $U_{d_2}$
as block diagonal matrices
followed by the action on $E_{d_2}$ corresponding to~$\gm^{(d)}.$
Then the standard inclusion of $E_{d_1}$ in~$E_{d_2}$
is equivariant.
\end{enumerate}
\end{thm}

\begin{proof}
The group action of~(\ref{T-PmOne-1}) is obtained as in
Remark 1.2(2) of~\cite{Pm}.
In~\cite{Pm}, for a general Hilbert bimodule~$F,$
only the action on the quotient $\OA{F}$
of ${\mathcal{T}}_F$ is described,
but the same reasoning also gives an action on ${\mathcal{T}}_F.$
Continuity of the action is easily checked on the
generators $T_{\xi}$ for $\xi \in \C^d,$
and continuity on the algebra follows by a standard argument.

For part~(\ref{T-PmOne-2}),
relations giving~${\mathcal{T}}_d$
as a universal \ca\  are described at the
beginning of Section~3 of~\cite{Pm}.
By comparing these relations with those for~$E_d,$
one sees that the maps $\sm_d$ exist and are are isomorphisms.

Part~(\ref{T-PmOne-3}) is a computation.
For part~(\ref{T-PmOne-4}),
orthogonality of the ranges of the~$s_j$
shows that if $\ld \in \C^{\infty}$
satisfies $\ld_j = 0$ for all but finitely many $j \in \N,$
then
\[
\left\| \sssum{j = 1}{\infty} \ld_j s_j \right\| = \| \ld \|_2.
\]
Since for all $k \in \N,$
\[
\sum_{j = 1}^{\infty} | u_{j, k} |^2 < \infty,
\]
this implies convergence on the right in the formula
in~(\ref{T-PmOne-4}).
The validity of the formula is now a computation
like that for part~(\ref{T-PmOne-3}).

Part~(\ref{T-PmOne-5}) now follows by comparing
the formulas for the actions from parts (\ref{T-PmOne-3})
and~(\ref{T-PmOne-4})
with the definitions of~$\sm_{d_1}$ and $\sm_{d_2}.$
\end{proof}

\begin{rmk}\label{R-IndentOI}
Following Theorem~\ref{T-PmOne}(\ref{T-PmOne-2}),
we will identify ${\mathcal{T}}_{\infty}$ with $E_{\infty} = \OI,$
and
for finite $d$ we will identify ${\mathcal{T}}_d$ with~$E_d.$
We then write the action $\gm^{(d)}$
of Theorem~\ref{T-PmOne}(\ref{T-PmOne-1})
as an action on~$E_d.$
\end{rmk}

\begin{dfn}\label{D-QFI}
Let $G$ be a topological group.
A {\emph{quasifree action}} of $G$ on~$\OI$
is an action of the form $\gm^{(\infty)} \circ \rh$
with $\gm^{(\infty)}$ as in Remark~\ref{R-IndentOI}
and for some \ct\  \hm\  $\rh \colon G \to U_{\infty}$
(that is, for some unitary representation~$\rh$
of $G$ on $l^2 (\N)$).
\end{dfn}

\begin{rmk}\label{R-IffFactor}
Theorem~\ref{T-PmOne}(\ref{T-PmOne-3})
implies that for $d \in \N,$
an action of a topological group $G$ on~$E_d$
is quasifree in the sense of \Lem{L:ExtQFree}
\ifo\  it factors through $\gm^{(d)} \colon U_d \to \Aut ( E_d )$
in a similar way.
\end{rmk}

\begin{prp}\label{P-UEImpConj}
Let $G$ be a topological group,
and let $\rh_1, \rh_2 \colon G \to U_{\infty}$ be unitarily equivalent
representations.
Then the corresponding quasifree actions of $G$ on~$\OI$
are conjugate.
\end{prp}

\begin{proof}
This is immediate from
parts (\ref{T-PmOne-1}) and~(\ref{T-PmOne-2}) of Theorem~\ref{T-PmOne}.
\end{proof}

\begin{thm}[Pimsner]\label{T-PmKKG}
Let $d \in \N \cup \{ \infty \},$
let $G$ be a second countable locally compact group,
and let $\rh \colon G \to U_d$ be a unitary representation.
With the action of $G$ on~$E_d$ as in Remark~\ref{R-IndentOI},
the inclusion of $\C$ in $E_d$ via $\ld \mapsto \ld \cdot 1$
is a $KK^G$-equivalence.
\end{thm}

\begin{proof}
See Theorem~4.4 and Remark 4.10(2) in~\cite{Pm}.
\end{proof}

\begin{dfn}\label{D-Filtered}
Let $G$ be a topological group,
and let $\rh \colon G \to U_{\infty}$
be a unitary representation of~$G$ on $l^2 (\N).$
We say that $\rh$ is {\emph{filtered}}
if there are $d (1) < d (2) < \cdots$ in $\N$ such that
for each~$k,$
the \pj\  $p_k$
on the span of the first $d (k)$ standard basis vectors in $l^2 (\N)$
is $G$-invariant.
We call the $d (k)$-dimensional representations $\rh_k$
given by $g \mapsto \rh (g) |_{p_k l^2 (\N)}$
the {\emph{filtering representations}}.
(They are, of course, not uniquely determined by~$\rh.$)

We say that a quasifree action
of $G$ on~$\OI$ {\emph{filtered}} if the corresponding
representation $\rh$ as in \Def{D-QFI} is filtered.
\end{dfn}

\begin{rmk}\label{R-FiltLim}
Let $G$ be a topological group,
and let $\af \colon G \to \Aut (\OI)$ be a quasifree action
of $G$ on~$\OI$
coming from a filtered representation $\rh \colon G \to U_{\infty}.$
Let the notation for a sequence of filtering representations
be as in \Def{D-Filtered}.
Let $\af^{(n)} \colon G \to \Aut (E_{d (n)})$
be the quasifree action of \Lem{L:ExtQFree}
coming from the representation $g \mapsto \rh (g) |_{p_k l^2 (\N)}.$
Then $\OI$ is the equivariant direct limit $\dirlim E_{d (n)}.$
This follows from Theorem~\ref{T-PmOne}
(using all its parts).
\end{rmk}

The following special version of a filtered representation
is introduced for technical convenience.

\begin{dfn}\label{D-EvenFilt}
Let $G$ be a topological group,
and let $\rh \colon G \to U_{\infty}$
be a filtered unitary representation of~$G$ on $l^2 (\N).$
We say that a collection $(\rh_k)_{n \in \N}$
of filtering representations is {\emph{almost even}}
if there exist $N_0, N \in \N$
and representations $\sm_0 \colon G \to U_{N_0}$
and $\sm \colon G \to U_{N}$ such that,
following the notation of \Def{D-Filtered},
\begin{enumerate}
\item\label{D-EvenFilt-1}
$d (n) = N_0 + n N$ for all $n \in \N.$
\item\label{D-EvenFilt-2}
$\rh_1 = \sm_0 \oplus \sm.$
\item\label{D-EvenFilt-3}
$\rh_{n + 1} = \rh_n \oplus \sm$ for all $n \in \N.$
\end{enumerate}
\end{dfn}

It is important that we have equality in
parts (\ref{D-EvenFilt-2}) and~(\ref{D-EvenFilt-3})
of \Def{D-EvenFilt},
not merely unitary equivalence.

\begin{lem}\label{L-ConjToFilt}
Let $G$ be a compact group,
and let $\af \colon G \to \Aut (\OI)$ be a quasifree action
of $G$ on~$\OI.$
Then:
\begin{enumerate}
\item\label{L-ConjToFilt-Cpt}
The action $\af$ is conjugate to a filtered quasifree action.
\item\label{L-ConjToFilt-Finite}
If $G$ is in fact finite,
then $\af$ is conjugate to the quasifree action
coming from a representation with an almost even filtration.
\end{enumerate}
\end{lem}

Part~(\ref{L-ConjToFilt-Finite}) can fail if the group
is not finite.
The regular representation of a second countable
infinite compact group does not have an almost even filtration.

\begin{proof}[Proof of \Lem{L-ConjToFilt}]
For both parts,
we use Proposition~\ref{P-UEImpConj}.

Part~(\ref{L-ConjToFilt-Cpt})
is immediate from the fact that every unitary representation of
a compact group is a direct sum of \fd\  representations.

For part~(\ref{L-ConjToFilt-Finite}),
we need to show that every representation $\pi \colon G \to U_{\infty}$
is unitarily equivalent to a representation~$\rh$
with an almost even filtration.
Let $\ta_1, \ta_2, \ldots, \ta_l$
be a set of representatives of the unitary equivalence classes
of irreducible representations of~$G.$
We may assume that $\ta_1, \ta_2, \ldots, \ta_{l_0}$
occur in $\pi$ with finite multiplicities
$m_1, m_2, \ldots, m_{l_0} \in \Nz,$
and that $\ta_{l_0 + 1}, \ta_{l_0 + 2}, \ldots, \ta_l$
occur with infinite multiplicity.
Then $l_0 < l.$
Take
\[
\sm_0 = \bigoplus_{k = 1}^{l_0} m_k \cdot \ta_k,
\,\,\,\,\,\,
\sm = \bigoplus_{k = l_0 + 1}^l \ta_k,
\andeqn
\rh = \sm_0 \oplus \sm \oplus \sm \oplus \cdots.
\]
This completes the proof.
\end{proof}

\begin{prp}\label{P:OIIsOuter}
Let $G$ be a topological group,
and let $\rh \colon G \to L (l^2 (\N))$
be an injective filtered unitary representation of~$G.$
Then the corresponding quasifree action
$\af \colon G \to \Aut ( \OI )$
is pointwise outer,
that is, $\af_g$ is outer for all $g \in G \setminus \{ 1 \}.$
\end{prp}

\begin{proof}
Adopt the notation of \Def{D-Filtered}.
Also let $\dt_1, \dt_2, \ldots$ be the standard basis vectors
of $l^2 (\N).$
Let $g \in G \setminus \{ 1 \}$;
we will show that $\af_g$ is outer.
Choose $k$ so large that $\rh (g) |_{p_k l^2 (\N)}$ is nontrivial.
Replacing $d (1), \, d (2), \ldots$ by $d (k), \, d (k + 1), \ldots,$
we may assume that $k = 1.$
Since $\rh (g) |_{p_1 l^2 (\N)}$ is unitary and nontrivial,
and since $p_1 l^2 (\N)$ is finite dimensional,
there exists a unitary $u \in L (l^2 (\N)),$
of the form $u = u_0 + (1 - p_1)$ with
$u_0$ a unitary in $L (p_1 l^2 (\N)),$
such that $\dt_1$ is an eigenvector of $u \rh (g) u^*$
with eigenvalue $\zt \neq 1.$
Let $\sm \colon G \to L (l^2 (\N))$ be the representation
$\sm (g) = u \rh (g) u^*,$
and let $\bt \colon G \to \Aut (\OI)$ be the corresponding
quasifree action.
It follows from Proposition~\ref{P-UEImpConj} that $\bt$
is conjugate to~$\af.$
Therefore it suffices to show that $\bt_g$ is outer.
Note that $\bt_g (s_1) = \zt s_1.$

We follow the proof of Theorem~4 of~\cite{ETW}.
Suppose $\bt_g$ is inner,
and let $v \in \OI$ be a unitary such that $\bt_g = \Ad (v).$
Define $f \colon \N \to \N \times \N$
by $f (j, l) = 2^{j - 1} (2 l - 1)$ for $j, l \in \N.$
Define isometries $t_j \in L (l^2 (\N))$ by
$t_j \dt_l = \dt_{f (j, l)}$ for $j, l \in \N.$
Since $f$ is injective,
there is a unital representation
$\pi \colon \OI \to L (l^2 (\N))$ such that
$\pi (s_j) = t_j$ for all $j \in \N.$
Since $\pi (v) \dt_1 \in l^2 (\N)$ and has norm~$1,$
we can write
$\pi (v) \dt_1 = \sum_{k = 1}^{\infty} \ld_k \dt_k$ with
$\sum_{k = 1}^{\infty} | \ld_k |^2 = 1.$
Computations similar to those in the proof of Theorem~4 of~\cite{ETW}
show that
\[
\sum_{k = 1}^{\infty} \ld_k \dt_k
  = \pi \big( \bt_g (s_1) \big) \pi (v) \dt_1
  = \sum_{k = 1}^{\infty} \zt \ld_k \dt_{2 k - 1}.
\]
Compare coefficients.
For $k = 1,$ we get $\ld_1 = 0$ since $\zt \neq 1.$
For $k > 1,$
we get
\[
\ld_k = \zt^{-1} \ld_{2 k - 1}
      = \zt^{-2} \ld_{2 (2 k - 1) - 1}
      = \cdots.
\]
Since $| \zt | = 1$ and $\sum_{l = 1}^{\infty} | \ld_l |^2 < \infty,$
this implies $\ld_k = 0.$
But then $\pi (v) \dt_1 = 0,$ a contradiction.
\end{proof}

\begin{lem}\label{L-CsqOfOuter}
Let $G$ be a topological group,
let $\rh \colon G \to L (l^2 (\N))$
be an injective filtered unitary representation of~$G,$
and let $\af \colon G \to \Aut ( \OI )$
be the corresponding quasifree action.
Then $(\OI)^G$ is purely infinite and simple,
and $K_1 \big( (\OI)^G \big) = 0.$
\end{lem}

\begin{proof}
Since $\af$ is pointwise outer
(Proposition~\ref{P:OIIsOuter}),
it follows from Theorem~3.1 of~\cite{Ks2-E1}
that $C^* (G, \OI, \af)$ is simple,
and from Corollary~4.6 of~\cite{JO}
that $C^* (G, \OI, \af)$ is purely infinite.
The Proposition in~\cite{Rs2} and its proof
imply that $(\OI)^G$ is isomorphic to
a corner in $C^* (G, \OI, \af),$
necessarily full.
Therefore $(\OI)^G$ is purely infinite and simple.

Theorem~\ref{T-PmKKG} implies that $K_1^G (\OI) = 0.$
{}From~\cite{Jg} or
Theorem 2.8.3(7) of~\cite{PhThs},
we get $K_1 \big( C^* (G, \OI, \af) \big) = 0.$
Since $(\OI)^G$ is a full corner in $C^* (G, \OI, \af),$
it follows from Proposition~1.2 of~\cite{Pk}
that $K_1 \big( (\OI)^G \big) = 0.$
\end{proof}

\begin{lem}\label{L-OrthInvIso}
Let $G$ be a finite group.
Let $\rh \colon G \to U_{\infty}$ be an injective representation with
an almost even filtration,
for which we use the notation of \Def{D-EvenFilt},
and let $\af^{(n)} \colon G \to \Aut (E_{d (n)})$
be as in Remark~\ref{R-FiltLim}.
For $m \in \N,$
set
\[
e_m = \sum_{j = N_0 + (m - 1) N + 1}^{N_0 + m N} s_j s_j^*.
\]
Then there exists $M \in \N$ such that for all $n \geq M,$
there are two isometries in $e_n (E_{d (n)})^G e_n$
with orthogonal ranges.
\end{lem}

\begin{proof}
Let $\af \colon G \to \Aut (\OI)$ be the
corresponding quasifree action
of $G$ on~$\OI.$
Following Remark~\ref{R-FiltLim},
we regard $E_{d (n)}$ as a subalgebra of~$\OI.$
\Lem{L-CsqOfOuter} implies
that $e_2 (\OI)^G e_2$ is purely infinite and simple.
It follows from \Lem{L-FixedPtLim}
that
\[
e_1 (\OI)^G e_1
 = {\overline{\bigcup_{n = 0}^{\infty} e_1 (E_{d (n)})^G e_1}}.
\]
Therefore there is $M \in \N$
such that there are isometries $t_1, t_2 \in e_1 (E_{d (M)})^G e_1$
with orthogonal ranges.

Now let $n \geq M.$
Recall from \Def{D-EvenFilt} that $\rh_n$ is the direct
sum of $\sm_0$ and $n$ copies of~$\sm.$
Let $u \in U_{N_0 + m N}$ be the permutation unitary
which exchanges the first and last copies of~$\sm.$
Then $u$ commutes with $\rh_n (g)$ for all $g \in G.$
Applying \Lem{L:ExtQFree} to the group $\Z \times G,$
we see that $u$ induces a quasifree automorphism $\ps$ of $E_{d (n)}$
which commutes with the action $\af^{(n)}.$
Moreover, $\ps (e_1) = e_n.$
Since $E_{d (M)} \subset E_{d (n)},$
the elements $\ps (t_1)$ and $\ps (t_2)$
are defined and are $G$-invariant isometries in $e_n (E_{d (n)})^G e_n$
with orthogonal ranges.
\end{proof}

The following result is the equivariant analog of
(a special case of) Lemma~3.3 of~\cite{Bl7}.
Our statement is more abstract;
the concrete version, analogous to that given in~\cite{Bl7},
is rather long.

\begin{lem}\label{L:OISjLem}
Let $G$ be a finite group,
let $\rh \colon G \to U_{\infty}$ be an injective representation with
an almost even filtration,
and let $\af \colon G \to \Aut (\OI)$ be the
corresponding quasifree action
of $G$ on~$\OI.$
Let the notation be as in
\Def{D-EvenFilt} and Remark~\ref{R-FiltLim}.
In particular, $\OI = \dirlim_n E_{d (n)}$;
call the maps of the system
\[
\io_{n, m} \colon E_{d (m)} \to E_{d (n)}
\andeqn
\io_{\infty, m} \colon E_{d (m)} \to \OI.
\]
Let $\GAa$ be a unital \ga,
and let $\pi \colon A \to \OI$ be a surjective \ehm.
Then there exists $M \in \N$ such that for all $n \geq M,$
the following holds.
Let $\ph \colon E_{d (n)} \to A$ be a unital \ehm\  such that
$\pi \circ \ph = \io_{\infty, n}.$
Then there exists a unital \ehm\  $\ps \colon E_{d (n + 1)} \to A$
such that
\[
\pi \circ \ps = \io_{\infty, \, n + 1}
\andeqn
\ps \circ \io_{n + 1, \, n - 1} = \ph \circ \io_{n, \, n - 1}.
\]
\end{lem}

Here is the diagram:
\[
\xymatrix{
& E_{d (n + 1)} \ar@{-->}[rd]_{\ps} \ar[rrrd]^{\io_{\infty, \, n + 1}}
                                    & & & \\
E_{d (n - 1)} \ar[ru]^{\io_{n + 1, \, n - 1}}
     \ar[rd]_{\io_{n, \, n - 1}} & & A \ar[rr]^(0.3){\pi} & & {\OI}. \\
& E_{d (n)} \ar[ru]^{\ph} \ar[rrru]_{\io_{\infty, \, n }} & & &
}
\]
The solid arrows are given,
and $\ps$ is supposed to exist which makes the
diagram commute.

\begin{proof}[Proof of \Lem{L:OISjLem}]
We use the names $r^{(d)}_j$
in Notation~\ref{N:ExtCuntz} for the generators
of~$E_d,$
and we denote the standard generators of $\OI$ by $s_1, s_2, \ldots.$
Also recall that $d (m) = N_0 + m N$ for $m \in \N.$

Choose $M$ as in \Lem{L-OrthInvIso}.

Define $e_0 = \sum_{j = 1}^{N_0} s_j s_j^*,$
which is the \pj\  in $\OI$ associated with the representation
$\sm_0 \colon G \to U_{N_0}.$
For $m \in \N,$
set
\[
e_m = \sum_{j = N_0 + (m - 1) N + 1}^{N_0 + m N} s_j s_j^*
\andeqn
q_m = \sum_{j = 1}^{N_0 + m N} s_j s_j^*.
\]
Thus, $e_m$ is the \pj\  in $\OI$ associated with the
$m$th copy of $\sm$ in the direct sum decomposition
\[
\rh = \sm_0 \oplus \sm \oplus \sm \oplus \cdots,
\]
and $q_m = \sum_{k = 0}^m e_k$ is similarly associated with~$\rh_m.$

For $k, l \in \N,$
define
\[
c_{k, l} = \sum_{j = 1}^N s_{d (k - 1) + j} (s_{d (l - 1) + j})^*.
\]
One easily checks that
\[
c_{k, l} c_{l, k} = e_k
\andeqn
c_{k, l} = c_{l, k}^*
\]
for $k, l \in \N.$
We claim that $c_{k, l}$ is $G$-invariant.
To prove the claim,
for $m \in \N$
let $\big( e_{j, k}^{(m)} \big)_{j, k = 1}^{d (m)}$
be the standard system of matrix units in $M_{d (m)},$
and let $\mu_m \colon M_{d (m)} \oplus \C \to E_{d (m)}$
be the \hm\  called $\mu_0$ in Notation~\ref{N-MdToOd}.
Recall that $E_{d (m)}$ has the action $\af^{(m)} = \af^{\rh_m},$
and equip $M_{d (m)} \oplus \C$ with the action
$\Ad ( \rh_m \oplus 1)$ (Notation~\ref{N-AdRep}).
Then $\mu_m$ is \eqv\  by \Lem{L:ExtQFree}(\ref{L:ExtQFree-2}).
Now let $g \in G.$
Set
\[
m = \max (k, l)
\andeqn
w = \sum_{j = 1}^N e_{d (k - 1) + j, \, d (l - 1) + j}
     \in M_{d (m)} \oplus \C.
\]
Then
$w$ is $G$-invariant since it is a partial isometry
which intertwines the $k$th and $l$th copies
of~$\sm$ in the direct sum decomposition of~$\rh_m.$
Therefore $c_{k, l} = (\io_{\infty, m} \circ \mu_m) (w)$
is also $G$-invariant.
The claim is proved.

Let $n \in \N$ satisfy $n \geq M.$
By the choice of~$M$ using \Lem{L-OrthInvIso},
there exist isometries
$t_1, t_2 \in e_n (E_{d (n)})^G e_n$ with orthogonal ranges.
Define partial isometries in $(\OI)^G$ by
\[
v_1 = \io_{\infty, n} (t_1)^*
\andeqn
v_2 = c_{n + 1, \, n} \io_{\infty, n} (t_2)^*.
\]
(For $G$-invariance of~$v_2,$
use the claim above.)
We now follow the proof of Lemma~3.3 of~\cite{Bl7}.
One checks that
\[
v_1 v_1^* = e_n,
\,\,\,\,\,\,
v_1^* v_1 = t_1 t_1^*,
\,\,\,\,\,\,
v_2 v_2^* = e_{n + 1},
\andeqn
v_2^* v_2 = t_2 t_2^*.
\]
Thus,
\[
q_{n - 1}, \,\, v_1^* v_1, \,\, v_2^* v_2
\andeqn
q_{n - 1}, \,\, v_1 v_1^*, \,\, v_2 v_2^*
\]
are two sets of \mops\  in~$(\OI)^G,$
and the \pj s
\[
1 - q_{n - 1} - v_1^* v_1 - v_2^* v_2
\andeqn
1 - q_{n - 1} - v_1 v_1^* - v_2 v_2^*
\]
are both nonzero and
have the same class in $K_0 \big( (\OI)^G \big).$
Therefore, by \Lem{L-CsqOfOuter},
we can find $v_3 \in (\OI)^G$
such that
\[
v_3^* v_3 = 1 - q_{n - 1} - v_1^* v_1 - v_2^* v_2
\andeqn
v_3 v_3^* = 1 - q_{n - 1} - v_1 v_1^* - v_2 v_2^*.
\]

Set
\[
w = v_1 + v_2 + v_3
\andeqn
v = q + v_1 + v_2 + v_3.
\]
Then $w$ is a unitary in $(1 - q_{n - 1}) (\OI)^G (1 - q_{n - 1}).$
Define
\[
p = \sum_{j = 1}^{d (n - 1)}
    \ph \big( r_{j}^{(d (n))} \big) \ph \big( r_{j}^{(d (n))} \big)^*,
\]
which is a $G$-invariant \pj\  in $A$
such that $\pi (p) = q_{n - 1}.$
Proposition~1.2 of~\cite{Pk} and \Lem{L-CsqOfOuter} imply that
\[
K_1 \big( (1 - q_{n - 1}) (\OI)^G (1 - q_{n - 1}) \big) = 0.
\]
Theorem~1.9 of~\cite{Cu2}
now implies that
$U \big( (1 - q_{n - 1}) (\OI)^G (1 - q_{n - 1}) \big)$ is connected.
The map $A^G \to (\OI)^G$ is surjective by \Lem{L:FixedPtQuot}.
So there exists a unitary $y \in (1 - p) A^G (1 - p)$
such that $\pi (y) = w.$
Set $u = p + y,$
which is a unitary in $A^G$ such that $\pi (u) = v.$

We have
$u \ph \big( r_{j}^{(d (n))} \big) = \ph \big( r_{j}^{(d (n))} \big)$
for $j = 1, 2, \ldots, d (n - 1).$
It is then easy to check that there is a unital \hm\  %
$\ps \colon E_{d (n + 1)} \to A$ satisfying
\[
\ps \big( r_{j}^{(d (n + 1))} \big) = \begin{cases}
   \ph \big( r_{j}^{(d (n))} \big)   &  j \leq d (n - 1) \\
   u \ph (t_1) \ph \big( r_{j}^{(d (n))} \big)
           &  d (n - 1) + 1 \leq j \leq d (n)
                                             \rule{0em}{3ex} \\
   u \ph (t_2) \ph \big( r_{j - N}^{(d (n))} \big)
           &  d (n) + 1 \leq j \leq d (n + 1).
                                             \rule{0em}{3ex}
\end{cases}
\]

Clearly
$\ps \circ \io_{n + 1, \, n - 1} = \ph \circ \io_{n + 1, \, n - 1}.$

For $j = 1, 2, \ldots, d (n),$
it is easily checked that
\[
(\pi \circ \ps) \big( r_{j}^{(d (n + 1))} \big)
  = \io_{\infty, \, n + 1} \big( r_{j}^{(d (n + 1))} \big).
\]
For $j = d (n) + 1, \, d (n) + 2, \, \ldots, \, d (n + 1),$
we have
\[
\pi \big( u \ph (t_2) \ph \big( r_{j - N}^{(d (n))} \big) \big)
  = v \io_{\infty, n} (t_2) s_{j - N}
  = c_{n + 1, \, n} \io_{\infty, n} (t_2^* t_2) s_{j - N}
  = c_{n + 1, \, n} s_{j - N}
  = s_j.
\]
It follows that $\pi \circ \ps = \io_{\infty, \, n + 1}.$

To finish the proof,
we must check that $\ps$ is equivariant.
It is enough to check equivariance on the generators.
Since $u$ is $G$-invariant and $\ph$ is \eqv,
it is enough to check equivariance of
the \hm\  $\ps_0 \colon E_{d (n + 1)} \to E_{d (n)}$
determined by
\[
\ps_0 \big( r_{j}^{(d (n + 1))} \big) = \begin{cases}
   r_{j}^{(d (n))}   &  j \leq d (n - 1) \\
   t_1 r_{j}^{(d (n))}
           &  d (n - 1) + 1 \leq j \leq d (n)
                                             \rule{0em}{3ex} \\
   t_2 r_{j - N}^{(d (n))}
           &  d (n) + 1 \leq j \leq d (n + 1).
                                             \rule{0em}{3ex}
\end{cases}
\]
Define $b, f \in E_{(d (n + 1))}$ by
\[
b = \sum_{j = 1}^N r^{(d (n + 1))}_{d (n) + j}
        \big( r^{(d (n + 1))}_{d (n - 1) + j} \big)^*
\andeqn
f = \sum_{j = 1}^{d (n - 1)} r^{(d (n + 1))}_{j}
        \big( r^{(d (n + 1))}_{j} \big)^*.
\]
Then $\io_{\infty, \, n + 1} (b) = c_{n + 1, \, n}$
and $\io_{\infty, \, n + 1}$ is injective and equivariant,
so $b$ is $G$-invariant.
Similarly, $\io_{\infty, \, n + 1} (f) = q_{n - 1},$
so $f$ is $G$-invariant.
Also, $b r^{(d (n + 1))}_{j - N} = r^{(d (n + 1))}_{j}$
for $j = d (n) + 1, \, d (n) + 2, \, \ldots, \, d (n + 1).$
Thus
\[
(\io_{n + 1, \, n} \circ \ps_0) \big( r_{j}^{(d (n + 1))} \big)
  = (f + t_1 + t_2 b) r_{j}^{(d (n + 1))}
\]
for $j = 1, \, 2, \, \ldots, \, d (n + 1).$
Since $f + t_1 + t_2 b$ is $G$-invariant,
and since $\io_{n + 1, \, n}$ is injective and \eqv,
it follows that $\ps_0$ is \eqv,
as desired.
This completes the proof.
\end{proof}

\begin{thm}\label{T:OIEsj}
Let $G$ be a finite group.
Let $\af \colon G \to \Aut (\OI)$ be a quasifree action.
Then $\af$ is \eqsj.
\end{thm}

\begin{proof}
We follow the proof of Theorem~3.2 of~\cite{Bl7}.
Let $\rh \colon G \to U_{\infty}$ be the representation which
gives rise to~$\af.$
Using \Lem{L-TrOnSbgp},
we may reduce to the case in which $\rh$ is injective.
By \Lem{L-ConjToFilt}(\ref{L-ConjToFilt-Finite}),
we may assume that $\rh$ has an almost even filtration
as in \Def{D-EvenFilt}.
Let the notation be as in \Lem{L:OISjLem},
and choose $M$ as there.

We follow the notation in Remark~\ref{R:EqSj}(\ref{R:EqSj:Nt}):
$C$ is a unital \ga\  with
an increasing sequence of invariant ideals $J_n,$
and $J = {\overline{\bigcup_{n = 1}^{\infty} J_n}}.$
The map $\pi_n \colon C / J_n \to C / J$ is the quotient map.

Let $\ph \colon \OI \to C / J$ be a unital \ehm.
First suppose that $\ph$ is an isomorphism.
{}From Theorem~\ref{T-QFEqSj} we get $n \in \N$
and a unital \ehm\  $\ps_M \colon E_{d (M)} \to C / J_n$
such that $\pi_n \circ \ps_M = \ph \circ \io_{\infty, M}.$
Applying \Lem{L:OISjLem} to
$\pi = \ph^{-1} \circ \pi_n \colon C / J_n \to \OI,$
for $m \geq M$
we inductively construct unital \ehm s
$\ps_m \colon E_{d (m)} \to C / J_n$
such that
\[
\pi \circ \ps_{m + 1} = \ph \circ \io_{\infty, m + 1}
\andeqn
\ps_{m + 1} \circ \io_{m + 1, \, m - 1} = \ps_m \circ \io_{m, m - 1}.
\]
Then $r_{j} = \limi{m} \ps_m \big( r^{d (m)}_j \big)$
exists for all $g \in G$ and $j \in \N,$
because when $d (m) \geq j$
it is equal to $\ps_{m + 2} \big( r^{d (m + 2)}_j \big).$
So there is a unital \ehm\  $\ps \colon \OI \to C / J_n$
such that $\ps (s_j) = r_j$ for all~$j \in \N.$
Clearly $\pi \circ \ps = \ph.$

For the general case,
set $Q = \ph (\OI) \subset C / J,$
let $D \subset C$ be the inverse image of~$Q,$
set $I_n = D \cap J_n$ for $n \in \N,$
and set $I = D \cap J.$
Then $I = {\overline{\bigcup_{n = 1}^{\infty} I_n}}.$
(In~\cite{Lr}, see Proposition 13.1.4, Lemma 13.1.5,
and the discussion afterwards.)
So $Q = D / I.$
Since $\OI$ is simple,
the corestriction $\ph_0 \colon \OI \to D / I$ of $\ph$
is an isomorphism.
The result follows by applying the special case above
with $D$ in place of~$C,$
with $I_n$ in place of~$J_n,$
and with $\ph_0$ in place of~$\ph.$
\end{proof}

\begin{pbm}\label{P-CptGpOI}
Let $G$ be an infinite compact group.
Is a quasifree action of $G$ on~${\mathcal{O}}_{\infty}$
necessarily \eqsj?
\end{pbm}

As a test case,
consider the quasifree action coming from the left regular
representation of~$S^1.$

\section{Equivariantly stable relations}\label{Sec:QqSR}

\indent
We relate equivariant semiprojectivity
to equivariant stability of relations
because, in the applications we have in mind~\cite{PhX},
equivariant stability of relations is what we actually use.

Weak stability of relations (Definition 4.1.1 of~\cite{Lr})
also has an equivariant version.
Since equivariant stability holds for the examples we care about,
we only consider equivariant stability.

We follow Section 13.2 of~\cite{Lr}
for our definition of generators and relations.

For reference,
we give the version of the definition without the group action,
except that we give a version for unital \ca s.
This is a variant of Definition 13.2.1 of~\cite{Lr}.

\begin{dfn}\label{D-NoGpRel}
Let $S$ be a set.
We denote by $F_S$
the universal \uca\  generated by the elements of $S$
subject to the relations $\| s \| \leq 2$ for all $s \in S.$
A {\emph{set of relations}} on $S$ is a subset $R \subset F_S.$
We refer to $(S, R)$ as a
{\emph{set of generators and relations}}.
We say that $(S, R)$ is {\emph{finite}} if $S$ and $R$ are finite.
We define $I_R \subset F_S$ to be the ideal
in $F_S$ generated by~$R.$
\end{dfn}

Since we are asking for unital algebras and \hm s,
we make the following definition.

\begin{dfn}\label{D-NoGpAdmi}
A set $(S, R)$ of generators and relations as in
\Def{D-NoGpRel} is {\emph{admissible}} if $I_R \neq F_S.$
When $(S, R)$ is admissible,
we let $\ta_R \colon F_S \to F_S / I_R$
be the quotient map.
The {\emph{C*-algebra on the generators and relations $(S, R)$}},
which we write $C^* (S, R),$
is by definition $F_S / I_R.$
We say that $(S, R)$ is {\emph{bounded}}
if for every $s \in S,$ we have $\| \ta_R (s) \| \leq 1.$
\end{dfn}

The choices $\| s \| \leq 2$ and $\| \ta_R (s) \| \leq 1$
are convenient normalizations.
By scaling, every set of generators and relations can
be fit in this framework.

The following is essentially Definition 13.2.2 of~\cite{Lr},
but for the unital situation.
By convention, we declare
(except in a few places where we explicitly allow it)
that the zero \ca\  is not unital.

\begin{dfn}\label{D-NoGpGen}
Let the notation be as in \Def{D-NoGpRel},
let $A$ be a unital \ca,
and let $\rh \colon S \to A$
be a function such that $\| \rh (s) \| \leq 2$
for all $s \in S.$
In this situation,
we write $\ph^{\rh} \colon F_S \to A$
for the corresponding \hm.
We say that $\rh$ is a {\emph{representation of $(S, R)$ in~$A$}}
if $\ph^{\rh} (x) = 0$ for all $x \in R.$
For $\dt \in [0, 1),$
we say that $\rh$ is a {\emph{$\dt$-representation}} of $(S, R)$
in~$A$
if $\| \ph^{\rh} (x) \| \leq \dt$
for all $x \in R.$
(Sometimes,
we will also allow the map to the zero \ca\  as a representation.)
If $(S, R)$ is admissible,
then the {\emph{universal representation}} $\rh_R$ is obtained by
taking $A = C^* (S, R)$ and $\rh_R = \ta_R |_S.$
\end{dfn}

\begin{rmk}\label{R-NoGpU}
It is clear that the universal representation,
as defined above,
really has the appropriate universal property.
\end{rmk}

\begin{lem}\label{L-D-NoGpAdmi}
Let $(S, R)$ be a set of generators and relations as in
\Def{D-NoGpRel}.
Then $(S, R)$ is admissible \ifo\  there exists
a representation in a (nonzero) unital \ca.
\end{lem}

\begin{proof}
This is immediate.
\end{proof}

\begin{rmk}\label{R-GeneralRel}
We make some general remarks.
\begin{enumerate}
\item\label{R-GeneralRel-1}
The relation corresponding to an element $x \in F_S$
is really just the statement $x = 0.$
Here $x$ could be any *-polynomial in the noncommuting
variables~$S,$
but in fact we are allowing arbitrary elements of the \ca~$F_S.$
The framework we describe in fact allows much more general
relations.
For example,
suppose $R_0 \subset F_S,$ $M \colon R_0 \to [0, \infty)$ is
a function, and we want the relations to
say $\| x \| \leq M (x)$
for all $x \in R_0.$
We simply take the intersection $I \subset F_S$ of the kernels of
all unital \hm s $\ph \colon F_S \to A,$
for arbitrary unital \ca s~$A,$
such that $\| \ph (x) \| \leq M (x)$ for all $x \in R_0.$
Then we take as relations all elements of~$I,$
that is, we take $R = I.$

Positivity conditions on elements of $F_S$ can be handled the same way.
\item\label{R-GeneralRel-2}
If $S$ is countable, we may always take $R$ to be finite.
Choose a countable subset $\{ x_1, x_2, \ldots \}$
of the unit ball of~$I_R$ whose span is dense in~$I_R.$
Then we can take the relations to consist of the single element
\[
a = \sum_{n = 1}^{\infty} 2^{- n} x_n^* x_n.
\]
(This change does, however, change the meaning
of a $\dt$-representation.)
\item\label{R-GeneralRel-3}
It follows from (\ref{R-GeneralRel-1})
and~(\ref{R-GeneralRel-2})
that if $(S, R)$ is finite and bounded,
and $\dt \in [0, 1),$
then the universal \ca\  generated by a $\dt$-representation
of $(R, S)$ is again the universal \ca\  on a finite and bounded
set of generators and relations.
\item\label{R-GeneralRel-4}
We have made a choice in the definition of a $\dt$-representation:
we still require $\| \rh (s) \| \leq 1$ for all $s \in S.$
By suitable scaling and application of (\ref{R-GeneralRel-1}) above,
it is also possible to get a version in which we merely require
$\| \rh (s) \| \leq 1 + \dt$ for all $s \in S.$
\end{enumerate}
\end{rmk}

We now give \eqv\  versions of these definitions.
We restrict to discrete groups,
and to finite groups in practice.
If $G$ is not discrete,
but the universal \ca\  is supposed to carry a \ct\  action
of~$G,$
then the relations must demand that the action of~$G$
on each generator defines a \cfn\  from $G$ to the universal \ca.
There are many kinds of conditions on elements of a \ca\  %
which can be made into relations which determine a universal \ca,
but continuity of functions from the set of generators isn't
one of them.
The universal algebra will in general only be an inverse limit of \ca s.
See Definition 1.3.4 and Proposition 1.3.6 of~\cite{Ph11}.
There do exist examples of universal $G$-algebras on generators and
relations when $G$ is not discrete.
See Example~\ref{E:NotCpt} and Example~\ref{E:NotCptMn} below.
However, we leave the development of the
appropriate theory for elsewhere.

\begin{ntn}\label{N-ActionOnFS}
Let $S$ be a set,
let $G$ be a discrete group,
and let $\sm$ be an action of $G$ on~$S,$
written $(g, s) \mapsto \sm_g (s).$
We denote by $\mu^{\sm}$ the action of $G$ on $F_S$
induced by~$\sm.$
\end{ntn}

\begin{dfn}\label{D:G-Rel}
Let $G$ be a discrete group.
A {\emph{$G$-equivariant set of generators and relations}}
is a triple $(S, \sm, R)$
in which $(S, R)$ is a set of generators and relations
as in \Def{D-NoGpRel},
$\sm$ is
an action of $G$ on $S$ (just as a set),
and $R$
is invariant under the action $\mu^{\sm}$
of Notation~\ref{N-ActionOnFS}.
We say that $(S, \sm, R)$ is {\emph{admissible}}
if $(S, R)$ is admissible
in the sense of \Def{D-NoGpAdmi}.
We say that $(S, \sm, R)$ is {\emph{bounded}} if $(S, R)$ is,
and is {\emph{finite}} if $G$ and $(S, R)$ are finite.
\end{dfn}

It may seem better to omit~$\sm$ and the requirement of
$G$-invariance,
and to allow the group action in the relations.
We address this formulation starting with
\Def{D:NoActR} below.
However, doing so does not give anything new,
and the version we have given above is technically more convenient.

\begin{dfn}\label{D:GRep}
Let $G$ be a discrete group.
Let $(S, \sm, R)$
be a $G$-equivariant set of generators and relations
in the sense of \Def{D:G-Rel}.
Let $\aGA$ be an action of $G$ on a \uca~$A.$
An {\emph{equivariant representation of $(S, \sm, R)$ in~$A$}}
is a representation
of $(S, R)$ in the sense of \Def{D-NoGpGen}
such that for every $g \in G$ and $s \in S,$
we have $\rh (\sm_g (s)) = \af_g (\rh (s)).$
For $\dt_1, \dt_2 \in [0, 1),$
a
{\emph{$\dt_1$-equivariant $\dt_2$-representation
of $(S, \sm, R)$ in~$A$}}
is a $\dt_2$-representation $\rh$ of $(S, R)$
such that $\| \rh (\sm_g (s)) - \af_g (\rh (s)) \| \leq \dt_1$
for all $g \in G$ and $s \in S.$
When $\dt_1 = 0,$
we speak of an {\emph{equivariant $\dt_2$-representation
of $(S, \sm, R)$ in~$A$}}.

If $(S, R)$ is admissible,
then the {\emph{universal equivariant representation}} $\rh_R$
is obtained by
taking $A = C^* (S, R),$
with the action ${\overline{\mu}}^{\sm} \colon G \to \Aut (C^* (S, R))$
coming from the fact that
$I_R$ is an invariant ideal
for $\mu^{\sm} \colon G \to \Aut (F_S),$
and taking $\rh_R = \ta_R |_S.$
We write $C^* (S, \sm, R)$
for the algebra equipped with this action.
\end{dfn}

We show that we have the right definition of admissibility.

\begin{lem}\label{L-GpAdmiN}
Let $G$ be a discrete group,
and let $(S, \sm, R)$ be
a $G$-equivariant set of generators and relations.
Then $(S, \sm, R)$ is admissible \ifo\  there exists
an \eqv\  representation in a (nonzero) unital \ga.
\end{lem}

\begin{proof}
If there is an \eqv\  representation,
then \Lem{L-D-NoGpAdmi} implies that $(S, R)$ is admissible,
so that $(S, \sm, R)$ is admissible.

For the reverse, since $I_R \neq F_S,$
the universal equivariant representation of \Def{D:GRep}
is an \eqv\  representation in a unital \ga.
\end{proof}

The universal equivariant representation,
as in \Def{D:GRep},
really is universal.

\begin{lem}\label{L-ReallyU}
Let $G$ be a discrete group,
and let $(S, \sm, R)$ be
a $G$-equivariant set of generators and relations.
Let $\aGA$ be an action of $G$ on a \uca~$A,$
and let $\rh \colon S \to A$
be an equivariant representation of $(S, \sm, R)$ in~$A.$
Then there exists a unique \ehm\  $\ph \colon C^* (S, \sm, R) \to A$
such that $\ph \circ \rh_R = \rh.$
\end{lem}

\begin{proof}
As an algebra, we have $C^* (S, \sm, R) = C^* (S, R).$
So Remark~\ref{R-NoGpU} provides
a unique \hm\  $\ph \colon C^* (S, \sm, R) \to A$
such that $\ph \circ \rh_R = \rh.$
Let ${\overline{\mu}}^{\sm} \colon G \to \Aut (C^* (S, \sm, R))$
be as in \Def{D:GRep}.
Since $\rh$ is \eqv,
for all $g \in G$ and $s \in S$ we have
\[
\ph \big( {\overline{\mu}}^{\sm}_g (\ta_R (s)) \big)
 = \ph (\ta_R (\mu^{\sm}_g (s)))
 = \rh (\sm_g (s))
 = \af_g (\rh (s))
 = \af_g (\ph \big( \ta_R (s) ) ).
\]
Since $\ta_R (S)$ generates $C^* (S, \sm, R),$
equivariance of $\ph$ follows.
\end{proof}

We have \eqv\  analogs of
the first two parts of Remark~\ref{R-GeneralRel}.

\begin{rmk}\label{R-GeneralEQRel}
\begin{enumerate}
\item\label{R-GeneralEqRel-1}
Let $G$ be a discrete group,
let $S$ be a set, and let $\sm$ be an action of $G$ on~$S.$
For any proper $G$-invariant ideal $I \subset F_S,$
we can get $F_S / I$ as a universal $G$-algebra
$C^* (S, \sm, R)$
simply by taking $R = I.$

As an example, let $R \subset F_S$ be $G$-invariant,
and let $M \colon R \to [0, \infty)$ be
a function such that $M (\sm_g (s)) = M (s)$ for all $g \in G$
and $s \in S.$
We take $I \subset F_S$ to be the intersection of the kernels of
all unital \ehm s $\ph \colon F_S \to A,$
for arbitrary unital \ga s $\GAa,$
such that $\| \ph (x) \| \leq M (x)$ for all $x \in R.$
\item\label{R-GeneralEqRel-2}
If $S$ is countable and $G$ is finite,
we always take $R$ to be finite.
Choose a countable subset $\{ x_1, x_2, \ldots \}$
of the unit ball of~$I_R$ whose span is dense in~$I_R.$
Then we can take the relations to consist of
the single $G$-invariant element
\[
a = \sum_{n = 1}^{\infty} \sum_{g \in G}
                 2^{- n} \mu^{\sm}_g (x_n^* x_n).
\]
\item\label{R-GeneralEqRel-3}
If $(S, \sm, R)$ is finite and bounded,
and $\dt \in [0, 1),$
then the universal \ca\  generated by an \eqv\  $\dt$-representation
of $(S, \sm, R)$ is again the universal \ca\  on a finite and bounded
set of generators and relations.
However,
for $\dt_0 > 0,$
there is no obvious action of~$G$
on the universal \ca\  generated by
a $\dt_0$-\eqv\  $\dt$-representation
of $(S, \sm, R).$
\end{enumerate}
\end{rmk}

If we want to allow the action of~$G$ to appear in the relations,
we can use the following alternate definition.
We omit the word ``equivariant'' in the name.

\begin{dfn}\label{D:NoActR}
Let $G$ be a discrete group.
A {\emph{set of generators and relations for a $G$-algebra}}
is a pair $(S, R)$
in which $S$ is a set
and $R$ is a subset of $F_{G \times S}.$
Define an action $\sm$ of $G$ on $G \times S$
by $\sm_g (h, s) = (g h, s)$ for $g, h \in G$ and $s \in S,$
and let $\mu^{\sm} \colon G \to \Aut (F_{G \times S} )$
be as in Notation~\ref{N-ActionOnFS}.
The {\emph{associated $G$-equivariant set of generators and relations}}
to $(S, R)$ is then
\[
\left( G \times S, \,\, \sm, \,\,
   {\ts{ {\ds{\bigcup}}_{g \in G} }} \mu^{\sm}_g (R) \right).
\]
We let $I_{G, R} \subset F_{G \times S}$
be the ideal generated by $\bigcup_{g \in G} \mu^{\sm}_g (R).$
We say that $(S, R)$ is {\emph{admissible}}
if $I_{G, R} \neq F_{G \times S},$
and in this case we define the
{\emph{universal $G$-algebra generated by $(S, R)$}}
to be $C^* (S, R) = F_{G \times S} / I_{G, R},$
with the action ${\overline{\mu}} \colon G \to \Aut (C^* (S, R))$
induced by the action $\mu^{\sm} \colon G \to \Aut (F_{G \times S}).$
Let $\ta_{G, R} \colon F_{G \times S} \to C^* (S, R)$
be quotient map.
We say that $(S, R)$ is {\emph{bounded}}
if for every $s \in S,$ we have $\| \ta_{G, R} (1, s) \| \leq 1.$
We say that $(S, R)$ is {\emph{finite}} if $G,$ $S,$ and~$R$
are all finite.
\end{dfn}

\begin{dfn}\label{D:RepNoAct}
Let $G$ be a discrete group,
and let $(S, R)$
be a set of generators and relations for a $G$-algebra
in the sense of \Def{D:NoActR}.
Let $\aGA$ be an action of $G$ on a \uca~$A.$
A {\emph{representation of $(S, R)$ in~$A$}}
is a function $\rh \colon S \to A$ such that
the function $\pi \colon G \times S \to A,$
defined by $\pi (g, s) = \af_g (\rh (s))$ for $g \in G$ and $s \in S,$
is an equivariant representation,
in the sense of \Def{D:GRep},
of the associated $G$-equivariant set of generators and relations.
For $\dt \in [0, 1),$
we say that $\rh$ is a {\emph{$\dt$-representation}} of $(S, R)$
in~$A$ if,
using the notation of \Def{D-NoGpGen},
we have $\| \ph^{\pi} (x) \| \leq \dt$
for all $x \in R.$
\end{dfn}

\begin{rmk}\label{R-CmpOfGpNoGp}
Let $G$ be a discrete group,
let $(S, R)$
be a set of generators and relations for a $G$-algebra
in the sense of \Def{D:NoActR},
and let the notation be as there.
Set $Q = \bigcup_{g \in G} \mu^{\sm}_g (R).$
Then:
\begin{enumerate}
\item\label{R-CmpOfGpNoGp-1}
$(S, R)$ is admissible \ifo\  $(G \times S, \, \sm, \, Q)$
is admissible in the sense of \Def{D:G-Rel}.
\item\label{R-CmpOfGpNoGp-2}
$(S, R)$ is bounded \ifo\  $(G \times S, \, \sm, \, Q)$
is bounded in the sense of \Def{D:G-Rel}.
(Use the fact that for $g \in G$ and $s \in S,$
we have $\ta_{G, R} (g, s) = {\overline{\mu}}_g ( \ta_{G, R} (1, s) ).$)
\item\label{R-CmpOfGpNoGp-3}
$(S, R)$ is finite \ifo\  $(G \times S, \, \sm, \, Q)$
is finite in the sense of \Def{D:G-Rel}.
\item\label{R-CmpOfGpNoGp-4}
There is a unique \eqv\  isomorphism
$\ps \colon C^* (S, R) \to C^* (G \times S, \, \sm, \, Q)$
such that $\ps ( \ta_{G, R} (g, s)) = \ta_R (g, s)$
for all $g \in G$ and $s \in S.$
\end{enumerate}
\end{rmk}

We then get the following universal property
for $C^* (S, R).$
The proof is clear,
and is omitted.

\begin{lem}\label{L-NoActUniv}
Let $G$ be a discrete group,
and let $(S, R)$ be
a set of generators and relations for a $G$-algebra
in the sense of \Def{D:NoActR}.
Let $\aGA$ be an action of $G$ on a \uca~$A,$
and let $\rh \colon S \to A$
be a representation of $(S, R)$ in~$A.$
Then there exists a unique \ehm\  $\ph \colon C^* (S, R) \to A$
such that $\ph \circ \rh_R = \rh.$
\end{lem}

\begin{rmk}\label{R-NoActGenR}
Analogously to
Remark~\ref{R-GeneralRel}(\ref{R-GeneralRel-1})
and Remark~\ref{R-GeneralEQRel}(\ref{R-GeneralEqRel-1}),
we can now speak of the universal \ga\  generated by a set~$S$
with relations given by norm bounds and positivity conditions
on *-polynomials in the noncommuting
variables $\coprod_{g \in G} \sm_g (S),$
that is, polynomials in the noncommuting
variables consisting of the generators, their formal adjoints,
and the formal images of all these under an action of~$G.$
\end{rmk}

We now present examples to show that there are
some cases in which there is a reasonable universal \ca,
with continuous action of~$G,$
even with $G$ not discrete.

\begin{exa}\label{E:NotCpt}
Let $G$ be any topological group,
and let $\GAa$ be any \ga.
Take the generating set $S$ to be the closed unit ball of~$A,$
take $R$ to be the collection of all algebraic relations
that hold among elements of $S$ and their adjoints,
and take $\sm = \af |_S.$
Then the universal \ca\  generated by $(S, \sm, R)$
is just $A,$
with the representation being the identity map
and the action of~$G$ being~$\af.$
The algebra~$A$ is universal when $G$ is given the discrete topology,
but the action is in fact \ct\  %
when $G$ is given its original topology.
\end{exa}

One can make a slightly more interesting example as follows.

\begin{exa}\label{E:NotCptMn}
Take $A = M_n,$
and take $\af$ to be any action of $G$ on~$M_n.$
Let $( e_{j, k} )_{j, k = 1}^n$ be the standard system
of matrix units in~$M_n.$
Take the generators to consist of elements
$v_{g, j, k}$ for $g \in G$ and $j, k \in \{ 1, 2, \ldots, n \}.$
Set $\sm_g (v_{h, j, k}) = v_{g h, j, k}.$
The universal representation is intended to be
$\rh (v_{g, j, k}) = \af_g (e_{j, k}).$
To make this happen,
take the relations to say that for each $g \in G,$
the collection $( v_{g, j, k} )_{j, k = 1}^n$
is a system of matrix units,
and also to include,
for all $g, h \in G$ and $j, k \in \{ 1, 2, \ldots, n \},$
the relation corresponding
to the (unique) expression of
$\af_g ( \af_h (e_{j, k}) )$ as a linear combination
of the matrix units $\af_h (e_{l, m}).$
\end{exa}

The following is the equivariant analog
of Definition 14.1.1 of~\cite{Lr}.
Following~\cite{Lr},
we restrict to finite sets of generators and relations.
Accordingly, we take the group to be finite.

\begin{dfn}\label{D:GStable}
Let $G$ be a finite group,
and let $(S, \sm, R)$
be a finite admissible $G$-equivariant set of generators and relations.
Then we say that $(S, \sm, R)$ is {\emph{stable}}
if for every $\ep > 0$ there is $\dt > 0$ such that the
following holds.
Suppose that $\GAa$ and $\GBb$ are unital \ga s
(except that we allow $B = 0$),
that $\om \colon A \to B$ is an \ehm,
that $\rh_0 \colon S \to A$
is a $\dt$-\eqv\  $\dt$-representation of $(S, \sm, R)$
(in the sense of \Def{D:GRep}),
and that $\om \circ \rh_0$
is an \eqv\  representation of $(S, \sm, R).$
Then there exists an \eqv\  representation $\rh \colon S \to A$
of $(S, \sm, R)$
such that $\om \circ \rh = \om \circ \rh_0$
and such that for all $s \in S$ we have
$\| \rh (s) - \rh_0 (s) \| < \ep.$
\end{dfn}

We allow $B = 0$ to incorporate the possibility that
we are merely given
a $\dt$-\eqv\  $\dt$-representation of $(S, \sm, R)$
but no \hm\  $\om$ such that $\om \circ \rh_0$ is an equivariant
representation.

\begin{lem}\label{L-MakeRepEq}
Let $G$ be a finite group,
and let $(S, \sm, R)$
be a bounded finite admissible
$G$-equivariant set of generators and relations.
Then for every $\et > 0$ there is $\dt > 0$
such that whenever $\GAa$ and $\GBb$ are unital \ga s
(with possibly $B = 0$),
$\om \colon A \to B$ is \eqv,
and $\rh_0 \colon S \to A$
is a $\dt$-\eqv\  $\dt$-representation of $(S, \sm, R)$
such that $\om \circ \rh_0$
is an \eqv\  representation of $(S, \sm, R),$
then there exists an (exactly) \eqv\  $\et$-representation
$\rh \colon S \to A$
such that $\om \circ \rh = \om \circ \rh_0$
and $\| \rh (s) - \rh_0 (s) \| < \et$
for all $s \in S.$
\end{lem}

\begin{proof}
Since $S$ and $R$ are finite,
there is $\dt_0 > 0$ such that whenever $C$ is a \ca\  %
and $\ps_1, \ps_2 \colon F_S \to C$ are two unital \hm s
such that $\| \ps_1 (s) - \ps_2 (s) \| < \dt_0$
for all $s \in S,$
then $\| \ps_1 (r) - \ps_2 (r) \| < \tfrac{1}{2} \et$
for all $r \in R.$
Set $\dt = \min \big( \dt_0, \tfrac{1}{2} \et \big).$

Now let $\rh_0$ be as in the hypotheses.
For $s \in S,$ define
\[
\rh (s) = \frac{1}{\card (G)} \sum_{g \in G}
            ( \af_g \circ \rh_0 \circ \sm_g^{-1} ) (s).
\]
Then $\rh$ is exactly \eqv.
Also, for all $s \in S,$
we have
\[
\| \rh (s) \|
 \leq \max \big( \big\{ \| \rh_0 (t) \| \colon t \in S \big\} \big)
 \leq 2
\]
and,
since $\rh_0$ is $\dt$-\eqv,
$\| \rh (s) - \rh_0 (s) \| \leq \dt \leq \dt_0.$
Therefore, in the notation of \Def{D-NoGpGen},
for all $r \in R$
we have
$\big\| \ph^{\rh} (r) - \ph^{\rh_0} (r) \big\| \leq \tfrac{1}{2} \et,$
whence
\[
\| \ph^{\rh} (r) \|
  \leq \tfrac{1}{2} \et + \big\| \ph^{\rh_0} (r) \big\|
  \leq \tfrac{1}{2} \et + \dt
  \leq \et.
\]
Thus $\rh$ is an $\et$-representation.
{}From
\[
\om \circ \af_g \circ \rh_0 \circ \sm_g^{-1}
 = \bt_g \circ \om \circ \rh_0 \circ \sm_g^{-1}
 = \om \circ \rh_0,
\]
we get $\om \circ \rh = \om \circ \rh_0,$
completing the proof.
\end{proof}

\begin{thm}\label{T:StabSj}
Let $G$ be a finite group,
and let $(S, \sm, R)$
be a bounded finite admissible
$G$-equivariant set of generators and relations.
Then $(S, \sm, R)$ is stable \ifo\  $C^* (S, \sm, R)$ is \eqsj.
\end{thm}

\begin{proof}
Proposition 13.2.5 of~\cite{Lr} holds equally well,
and with the same proof,
for unital algebras,
for a bounded finite admissible
$G$-equivariant set $(S, \sm, R)$ of generators and relations
(with, in particular, $G$ finite),
for an \eqv\  direct system of unital \ga s
with unital maps,
and for a $\dt$-\eqv\  $\dt$-representation of $(S, \sm, R).$
Therefore stability of $(S, \sm, R)$
implies equivariant semiprojectivity of $C^* (S, \sm, R).$

The prooof of the reverse implication
roughly follows the proof
for the non\eqv\  case,
as, for example, in the proof of Theorem 14.1.4 of~\cite{Lr}.
For $n \in \N$ let $J_n \subset F_S$
be the intersection of the kernels
of the \hm s $\ph^{\rh}$ as $\rh$ runs through
all \eqv\  $2^{-n}$-representations of $(S, \sm, R).$
Then $J_n$ is a $G$-invariant ideal in~$F_S,$
\[
J_1 \subset J_2 \subset \cdots,
\andeqn
{\overline{\bigcup_{n = 1}^{\infty} J_n}} = I_R.
\]
The quotient $F_S / J_n$ is the universal $G$-algebra
generated by an \eqv\  $2^{-n}$-representation of $(S, \sm, R).$
We will apply the definition of equivariant semiprojectivity
to $C^* (S, \sm, R),$
with $C = F_S,$
with $J_n$ as given,
with $J = I_R,$
and with $\ph = \id_{C^* (S, \sm, R)}.$
We use the same names
$\kp \colon F_S \to F_S / I_R,$
$\kp_n \colon F_S \to F_S / J_n,$
$\pi_n \colon F_S / J_n \to F_S / I_R,$
etc.\  %
for the maps as in Definition~\ref{D:EqSj}
and Remark~\ref{R:EqSj}(\ref{R:EqSj:Nt}).
By equivariant semiprojectivity, we can choose $n_0 \in \N$
and a unital \ehm\  %
$\ps_0 \colon C^* (S, \sm, R) \to F_S / J_{n_0}$
such that $\pi_{n_0} \circ \ps_0 = \id_{C^* (S, \sm, R)}.$

For $s \in S$ we have
\[
(\pi_{n_0} \circ \ps_0 \circ \kp) (s)
  = (\pi_{n_0} \circ \ps_0 \circ \pi_{n_0} \circ \kp_{n_0}) (s)
  = (\pi_{n_0} \circ \kp_{n_0}) (s).
\]
Since $S$ is finite,
there is $n \geq n_0$ such that
for all $s \in S$ we have
\[
\big\| (\pi_{n, n_0} \circ \ps_0 \circ \kp) (s)
        - \kp_{n} (s) \big\|
   = \big\| (\pi_{n, n_0} \circ \ps_0 \circ \kp) (s)
        - (\pi_{n, n_0} \circ \kp_{n_0}) (s) \big\|
   < \tfrac{1}{2} \ep.
\]
We may also require that $2^{- n} < \tfrac{1}{2} \ep.$
Define $\ps = \pi_{n, n_0} \circ \ps_0,$
getting
$\pi_{n} \circ \ps = \id_{C^* (S, \sm, R)}$
and
\begin{equation}\label{Eq:PsKpStab}
\| (\ps \circ \kp) (s) - \kp_{n} (s) \| < \tfrac{1}{2} \ep
\end{equation}
for all $s \in S.$

Choose $\dt > 0$ as in \Lem{L-MakeRepEq} for $\et = 2^{- n}.$
Let $\GAa$ and $\GBb$ be unital \ga s
(with possibly $B = 0$),
let $\om \colon A \to B$ be \eqv,
and let $\rh_0 \colon S \to A$
be a $\dt$-\eqv\  $\dt$-representation of $(S, \sm, R)$
such that $\om \circ \rh_0$
is an \eqv\  representation of $(S, \sm, R).$
By the choice of~$\dt,$
there is an \eqv\  $2^{-n}$-representation
$\rh_1 \colon S \to A$
such that $\om \circ \rh_1 = \om \circ \rh_0$
and
\begin{equation}\label{Eq:Rh0Rh1Stab}
\| \rh_1 (s) - \rh_0 (s) \| < 2^{-n}
\end{equation}
for all $s \in S.$

The following diagram (in which the triangle and the square
will be shown to commute,
and we already know that $\pi_n \circ \kp_n = \kp$)
shows some of the maps we have or which will be constructed:
\[
\xymatrix{
S \ar[r] \ar[rrd]_{\rh_1}
    & F_S \ar[r]^-{\kp_n} \ar@/^2.5pc/[rr]_{\kp}
    & F_S / J_n \ar[r]_{\pi_n} \ar[d]_{\ph}
    & F_S / I_R \ar[d]^{\ld} \ar@/_1.0pc/@{-->}[l]_(0.6){\ps}   \\
& & A \ar[r]^{\om}& B.
}
\]

By the definition of~$J_n,$
there is a unital \ehm\  $\ph \colon F_S / J_n \to A$
such that $\ph (\kp_n (s)) = \rh_1 (s)$ for all $s \in S.$
Define $\rh (s) = (\ph \circ \ps \circ \kp) (s)$
for $s \in S.$
Then $\rh$ is an \eqv\  representation of $(S, \sm, R).$
Moreover, there is an \ehm\  $\ld \colon F_S / I_R \to B$
such that $\ld ( \kp (s)) = \om (\rh_1 (s))$
for all $s \in S.$
By construction,
for $s \in S$ we have
\[
(\om \circ \ph \circ \kp_n) (s)
  = ( \om \circ \rh_1) (s)
  = (\ld \circ \pi_n \circ \kp_n) (s).
\]
Since $\kp_n$ is surjective and $S$ generates $F_S,$
we get
$\om \circ \ph = \ld \circ \pi_n.$
For $s \in S$ we now have
\begin{align*}
(\om \circ \rh) (s)
& = (\om \circ \ph \circ \ps \circ \kp) (s)
  = (\ld \circ \pi_n \circ \ps \circ \kp) (s)
         \\
& = (\ld \circ \kp) (s)
  = (\om \circ \rh_1) (s)
  = (\om \circ \rh_0) (s).
\end{align*}
That is, $\om \circ \rh = \om \circ \rh_0.$

It remains only to show that $\| \rh (s) - \rh_0 (s) \| < \ep$
for $s \in S.$
Using~(\ref{Eq:Rh0Rh1Stab}) and $2^{-n} < \tfrac{1}{2} \ep$
at the second step,
we have
\[
\| \rh (s) - \rh_0 (s) \|
  \leq \| \rh (s) - \rh_1 (s) \| + \| \rh_1 (s) - \rh_0 (s) \|
  < \| \rh (s) - \rh_1 (s) \| + \tfrac{1}{2} \ep,
\]
and, by~(\ref{Eq:PsKpStab}),
\[
\| \rh (s) - \rh_1 (s) \|
  = \| (\ph \circ \ps \circ \kp) (s) - (\ph \circ \kp_n) (s) \|
  \leq \| (\ps \circ \kp) (s) - \kp_n (s) \|
  < \tfrac{1}{2} \ep.
\]
The required estimate follows,
and the theorem is proved.
\end{proof}

We now consider the version of stability
in which the group action is allowed
in the relations.

\begin{dfn}\label{D-NoActStab}
Let $G$ be a finite group,
and let $(S, R)$ be a finite admissible
set of generators and relations for a $G$-algebra,
in the sense of \Def{D:NoActR}.
We say that $(S, R)$ is {\emph{stable}}
if for every $\ep > 0$ there is $\dt > 0$ such that the
following holds.
Suppose that $\GAa$ and $\GBb$ are unital \ga s
(except that we allow $B = 0$),
that $\om \colon A \to B$ is an \ehm,
that $\rh_0 \colon S \to A$
is a $\dt$-representation of $(S, R)$
(in the sense of \Def{D:RepNoAct}),
and that $\om \circ \rh_0$
is a representation of $(S, R).$
Then there exists a representation $\rh \colon S \to A$
of $(S, R)$
such that $\om \circ \rh = \om \circ \rh_0$
and such that for all $s \in S$ we have
$\| \rh (s) - \rh_0 (s) \| < \ep.$
\end{dfn}

\begin{lem}\label{L-EqVsNEdtRep}
Let $G$ be a finite group,
and let $(S, R)$
be a finite bounded admissible
set of generators and relations for a $G$-algebra
in the sense of \Def{D:NoActR}.
Let the action $\sm$ of $G$ on $G \times S$
be as there, and set
$Q = \bigcup_{g \in G} \mu^{\sm}_g (R),$
so that the associated $G$-equivariant set of generators and relations
is $(G \times S, \, \sm, \, Q).$
Then:
\begin{enumerate}
\item\label{L-EqVsNEdtRep-1}
For every $\et > 0$ there is $\dt > 0$ such that whenever
$\GAa$ is a unital \ga\  and $\ld \colon G \times S \to A$
is a $\dt$-equivariant $\dt$-representation of
$(G \times S, \, \sm, \, Q),$
then the function $s \mapsto \ld (1, s)$ is an
$\et$-representation of $(S, R).$
\item\label{L-EqVsNEdtRep-2}
For every $\et > 0$ there is $\dt > 0$ such that whenever
$\GAa$ is a unital \ga\  and $\rh \colon S \to A$
is a $\dt$-representation of $(S, R),$
then the function $(g, s) \mapsto \af_g (\rh (s))$
is an equivariant $\et$-representation of $(G \times S, \, \sm, \, Q).$
\end{enumerate}
\end{lem}

\begin{proof}
We prove part~(\ref{L-EqVsNEdtRep-1}).
Suppose the conclusion fails.
Apply \Def{D:RepNoAct} and use finiteness of $R$
to find $x \in R,$ $\et > 0,$
and for each $n \in \N$ a unital \ga\  $\big( G, A_n, \af^{(n)} \big)$
and a $\frac{1}{n}$-equivariant $\frac{1}{n}$-representation
$\ld_n \colon G \times S \to A_n$
such that,
if we define $\rh_n (s) = \ld_n (1, s)$ for $s \in S$
and $\pi_n (g, s) = \af^{(n)}_g ( \rh_n (s))$
for $g \in G$ and $s \in S,$
then, following the notation of \Def{D-NoGpGen},
we have $\| \ph^{\pi_n} (x) \| > \et.$

Let $\prod_{n = 1}^{\infty} A_n$ be the C*-algebraic product
(the set of sequences $(a_n)_{n \in \N}$ in the algebraic
product such that $\sup_{n \in \N} \| a_n \|$ is finite),
and define
\[
A = \left. \prod_{n = 1}^{\infty} A_n \middle/
    \bigoplus_{n = 1}^{\infty} A_n. \right.
\]
The obvious coordinatewise definitions,
followed by the quotient map,
give an action $\af \colon G \to \Aut (A)$
and functions
\[
\ld \colon G \times S \to A,
\,\,\,\,\,\,
\rh \colon S \to A,
\andeqn
\pi \colon G \times S \to A.
\]
One checks that $\ld$ is an equivariant representation of
$(G \times S, \, \sm, \, Q).$
Clearly $\rh (s) = \ld (1, s)$ for $s \in S$
and $\pi (g, s) = \af_g ( \rh (s))$ for $g \in G$ and $s \in S.$
Therefore $\pi = \ld.$
Since $x \in Q,$ we have $\ph^{\pi} (x) = 0.$
This contradicts the fact that $\| \ph^{\pi_n} (x) \| > \et$
for all $n \in \N.$
Part~(\ref{L-EqVsNEdtRep-1}) is proved.

Now suppose part~(\ref{L-EqVsNEdtRep-2}) is false.
Since $Q$ is finite,
there exist $x \in Q,$ $\et > 0,$
and for each $n \in \N$ a unital \ga\  $\big( G, A_n, \af^{(n)} \big)$
and a $\frac{1}{n}$-representation
$\rh_n \colon S \to A_n$
such that,
if we define $\pi_n (g, s) = \af^{(n)}_g ( \rh_n (s))$
for $g \in G$ and $s \in S,$
then $\| \ph^{\pi_n} (x) \| > \et.$
The functions $\pi_n$ are equivariant.
Define $A,$ $\af,$ $\rh,$ and $\pi$ as in the proof
of part~(\ref{L-EqVsNEdtRep-1}).
Then $\rh$ is a representation of $(S, R),$
$\pi$ is an equivariant representation of
$(G \times S, \, \sm, \, Q),$
and $\pi (g, s) = \af_g ( \rh (s))$ for all $g \in G$ and $s \in S.$
Therefore $\ph^{\pi} (x) = 0,$
contradicting $\| \ph^{\pi_n} (x) \| > \et$
for all $n \in \N.$
\end{proof}

\begin{thm}\label{T-StSjNoAction}
Let $G$ be a finite group,
and let $(S, R)$
be a bounded finite admissible
set of generators and relations for a $G$-algebra.
Then $C^* (S, R)$ is \eqsj\  %
\ifo\   $(S, R)$ is stable in the sense of Definition~\ref{D-NoActStab}.
\end{thm}

\begin{proof}
Define $\sm,$ $\mu,$ and~$Q$ as in \Lem{L-EqVsNEdtRep}.
It follows from Remark~\ref{R-CmpOfGpNoGp}
that $(G \times S, \, \sm, \, Q)$ is bounded, finite, and admissible.
By Theorem~\ref{T:StabSj}
and Remark~\ref{R-CmpOfGpNoGp}(\ref{R-CmpOfGpNoGp-4}),
it therefore suffices to prove that $(S, R)$ is stable \ifo\  %
$(G \times S, \, \sm, \, Q)$ is stable in the sense of \Def{D:GStable}.

Assume that $(S, R)$ is stable.
Let $\ep > 0.$
Choose $\et > 0$ as in \Def{D-NoActStab}
(where the number is called~$\dt$),
for $\tfrac{1}{2} \ep$ in place of~$\ep.$
Choose $\dt_0 > 0$
following \Lem{L-EqVsNEdtRep}(\ref{L-EqVsNEdtRep-1})
(where the number is called~$\dt$).
We may also require that $\dt_0 \leq \tfrac{1}{2} \ep.$
Apply \Lem{L-MakeRepEq},
with $\dt_0$ in place of~$\et,$
to get a number $\dt > 0.$

Let $\GAa$ and $\GBb$ be unital \ga s
(except that we allow $B = 0$),
and let $\om \colon A \to B$ be an \ehm.
Let $\ld_0 \colon G \times S \to A$ be
a $\dt$-\eqv\  $\dt$-representation of $(G \times S, \, \sm, \, Q)$
such that $\om \circ \ld_0$
is an \eqv\  representation of $(S, \sm, R).$
By the choice of $\dt,$
there is an \eqv\  $\dt_0$-representation
$\ld_1 \colon G \times S \to A$
such that $\om \circ \ld_1 = \om \circ \ld_0$
and $\| \ld_1 (g, s) - \ld_0 (g, s) \| < \dt_0$
for all $g \in G$ and $s \in S.$

Define $\rh_1 \colon S \to A$
by $\rh_1 (s) = \ld_1 (1, s)$ for $s \in S.$
Since $\ld_1$ is \eqv,
$\rh_1$ is a $\dt_0$-representation of $(S, R).$
Clearly $\om \circ \rh_1$
is a representation of $(S, R).$
By the choice of~$\dt_0,$
there exists a representation $\rh \colon S \to A$
of $(S, R)$
such that $\om \circ \rh = \om \circ \rh_1$
and such that for all $s \in S$ we have
$\| \rh (s) - \rh_1 (s) \| < \tfrac{1}{2} \ep.$
Define $\ld \colon G \times S \to A$ by
$\ld (g, s) = \af_g ( \rh (s))$ for $g \in G$ and $s \in S.$
Then, using equivariance of $\om$ at the first step,
we have
$\om \circ \ld = \om \circ \ld_1 = \om \circ \ld_0.$
Moreover, for $g \in G$ and $s \in S,$
by equivariance of $\ld$ and~$\ld_1,$
we have
$\ld (g, s) = \af_g (\rh (s))$ and $\ld_1 (g, s) = \af_g (\rh_1 (s)).$
Therefore
\[
\| \ld (g, s) - \ld_0 (g, s) \|
 \leq \| \af_g ( \rh (s)) - \af_g ( \rh_1 (s)) \|
        + \| \ld_1 (g, s) - \ld_0 (g, s) \|
 < \tfrac{1}{2} \ep + \dt_0
 \leq \ep.
\]
This completes the proof that
$(G \times S, \, \sm, \, Q)$ is stable.

For the reverse, assume that $(G \times S, \, \sm, \, Q)$ is stable.
We prove that $(S, R)$ is stable.
Let $\ep > 0.$
Choose $\et > 0$ as in \Def{D:GStable}
(where the number is called~$\dt$).
Choose $\dt > 0$ as in \Lem{L-EqVsNEdtRep}(\ref{L-EqVsNEdtRep-2}).
Let $\GAa$ and $\GBb$ be unital \ga s
(except that we allow $B = 0$),
and let $\om \colon A \to B$ be an \ehm.
Let $\rh_0 \colon S \to A$
be a $\dt$-representation of $(S, R)$
such that $\om \circ \rh_0$
is a representation of $(S, R).$
Define $\pi_0 \colon G \times S \to A$ by
$\pi_0 (g, s) = \af_g (\rh (s))$ for $g \in G$ and $s \in S.$
Then $\pi_0$
is an equivariant $\dt$-representation of $(G \times S, \, \sm, \, Q).$
Therefore there exists an \eqv\  representation
$\pi \colon G \times S \to A$
of $(S, \sm, R)$
such that $\om \circ \pi = \om \circ \pi_0$
and such that for all $g \in G$ and $s \in S$ we have
$\| \pi (g, s) - \pi_0 (g, s) \| < \ep.$
Define $\rh \colon S \to A$
by $\rh (s) = \pi (1, s)$ for $s \in S.$
Then $\| \rh (s) - \rh_0 (s) \| < \ep$
for all $s \in S.$
Also clearly $\om \circ \rh = \om \circ \rh_0.$
This completes the proof of the theorem.
\end{proof}

As an immediate application,
we can derive stronger versions of the Rokhlin property
for actions of finite groups
(Definition~3.1 of~\cite{Iz};
formulated without the central sequence algebra in
Definition~1.1 of~\cite{PhT1})
and the tracial Rokhlin property
(Definition~1.2 of~\cite{PhT1}).

\begin{prp}\label{P-RkP}
Let $A$ be a separable unital \ca,
and let $\af \colon G \to \Aut (A)$
be an action of a finite group $G$ on $A.$
Then $\af$ has the
Rokhlin property \ifo\  for every finite set
$F \subset A$ and every $\ep > 0,$
there are \mops\  $e_g \in A$ for $g \in G$ such that:
\begin{enumerate}
\item\label{P-RkP-1}
$\af_g (e_h) = e_{g h}$ for all $g, h \in G.$
\item\label{P-RkP-2}
$\| e_g a - a e_g \| < \ep$ for all $g \in G$ and all $a \in F.$
\item\label{P-RkP-3}
$\sum_{g \in G} e_g = 1.$
\end{enumerate}
\end{prp}

The definition of the Rokhlin property differs in that
in condition~(\ref{P-RkP-1}),
one merely requires
$\| \af_g (e_h) - e_{g h} \| < \ep$ for all $g, h \in G.$

The proof is very similar to, but simpler than,
the proof of Proposition~\ref{P-TRP},
and is omitted.

\begin{prp}\label{P-TRP}
Let $A$ be an infinite dimensional simple separable
unital \ca,
and let $\af \colon G \to \Aut (A)$
be an action of a finite group $G$ on $A.$
Then $\af$ has the
tracial Rokhlin property \ifo\  for every finite set
$F \subset A,$ every $\ep > 0,$
and every positive element $x \in A$ with $\| x \| = 1,$
there are \mops\  $e_g \in A$ for $g \in G$ such that:
\begin{enumerate}
\item\label{P-TRP-1}
$\af_g (e_h) = e_{g h}$ for all $g, h \in G.$
\item\label{P-TRP-2}
$\| e_g a - a e_g \| < \ep$ for all $g \in G$ and all $a \in F.$
\item\label{P-TRP-3}
With $e = \sum_{g \in G} e_g,$ the \pj\  $1 - e$ is \mvnt\  to a
\pj\  in the \hsa\  of $A$ generated by~$x.$
\item\label{P-TRP-4}
With $e$ as in~(\ref{P-TRP-3}), we have $\| e x e \| > 1 - \ep.$
\setcounter{TmpEnumi}{\value{enumi}}
\end{enumerate}
\end{prp}

The definition of the tracial Rokhlin property differs in that
in condition~(\ref{P-TRP-1})
one merely requires
$\| \af_g (e_h) - e_{g h} \| < \ep$ for all $g, h \in G.$

We give the details of the proof to demonstrate how our machinery works,
and in particular to show why
we do not want to require our $\dt$-representations to
be exactly \eqv.

\begin{proof}[Proof of Proposition~\ref{P-TRP}]
Let $F \subset A$ be finite and let $\ep > 0.$
By scaling, \wolog\  $\| a \| \leq 1$ for all $a \in F.$
Set $n = \card (G)$ and
\[
\ep_0 = \min \left( \frac{1}{n}, \, \frac{\ep}{2 n + 1} \right).
\]

Let $S$ consist of distinct elements $p_g$ for $g \in G.$
Define an action $\sm$ of $G$ on~$S$
by $\sm_g (p_h) = p_{g h}$ for $g, h \in G.$
Define
\[
R = \big\{ p_g p_h - \dt_{g, h} p_g \colon g, h \in G \big\}
      \cup \big\{ p_g^* - p_g \colon g \in G \big\}.
\]
Then $(S, \sm, R)$ is an equivariant
set of generators and relations,
and $C^* (S, \sm, R)$ is \eqv ly isomorphic to $C (G) \oplus \C,$
with the action on $C (G)$
coming from the translation action of~$G$ on itself
and the trivial action on~$\C,$
in such a way that $p_g$ is sent to $(\ch_{\{ g \}}, 0).$
This action is \eqsj\  by Theorem~\ref{T:FDEqSj}.
So $(S, \sm, R)$ is stable by Theorem~\ref{T:StabSj}.
Choose $\dt_0 > 0$ as in the definition of stability
for $\ep_0$ in place of~$\ep.$
Set $\dt = \min ( \dt_0, \ep_0 ).$
Apply the tracial Rokhlin property with $\dt$ in place of~$\ep,$
obtaining \mops\  $e^{(0)}_g \in A$ for $g \in G$ such that:
\begin{enumerate}
\setcounter{enumi}{\value{TmpEnumi}}
\item\label{P-PfRkP-1}
$\big\| \af_g \big( e^{(0)}_h \big) - e^{(0)}_{g h} \big\| < \dt$
for all $g, h \in G.$
\item\label{P-PfRkP-2}
$\big\| e_g^{(0)} a - a e_g^{(0)} \big\| < \dt$
for all $g \in G$ and all $a \in F.$
\item\label{P-PfRkP-3}
With $e^{(0)} = \sum_{g \in G} e_g^{(0)},$ the \pj\  $1 - e^{(0)}$
is \mvnt\  to a
\pj\  in the \hsa\  of $A$ generated by~$x.$
\item\label{P-PfRkP-4}
With $e^{(0)}$ as in~(\ref{P-PfRkP-3}),
we have $\big\| e^{(0)} x e^{(0)} \big\| > 1 - \ep_0.$
\end{enumerate}

Define $\rh_0 \colon S \to A$
by $\rh_0 (p_g) = e^{(0)}_g$ for $g \in G.$
Then $\rh_0$ is a
$\dt$-\eqv\  $\dt$-representation of $(S, \sm, R).$
Therefore there is an \eqv\  representation $\rh$ of $(S, \sm, R)$
such that $\big\| \rh (p_g) - e^{(0)}_g \big\| < \ep_0$
for all $g \in G.$
Set $e_g = \rh (p_g)$ for $g \in G.$
By the definition of an \eqv\  representation,
the $e_g$ are \mops\  %
satisfying condition~(\ref{P-TRP-1}).
Condition~(\ref{P-TRP-2})
follows from the estimates
\[
\| a \| \leq 1,
\,\,\,\,\,
\big\| e_g - e^{(0)}_g \big\| < \ep_0 \leq \tfrac{1}{3} \ep,
\andeqn
\big\| e_g^{(0)} a - a e_g^{(0)} \big\|
 < \dt \leq \ep_0 \leq \tfrac{1}{3} \ep
\]
for $g \in G$ and $a \in F.$

It remains to prove conditions (\ref{P-TRP-3}) and~(\ref{P-TRP-4}).
Set $e = \sum_{g \in G} e_g.$
First,
we have
\[
\big\| e - e^{(0)} \big\| < n \ep_0.
\]
Since $n \ep_0 \leq 1,$
the \pj\  $e$ is \mvnt\  to~$e^{(0)},$
and therefore also to a
\pj\  in the \hsa\  of $A$ generated by~$x.$
This is~(\ref{P-TRP-3}).
Moreover,
\[
\| e x e \|
 \geq \big\| e^{(0)} x e^{(0)} \big\| - 2 \big\| e - e^{(0)} \big\|
 > 1 - \ep_0 - 2 n \ep_0
 \geq 1 - \ep.
\]
This is~(\ref{P-TRP-4}).
\end{proof}

\section{Graded semiprojectivity of the C*-algebra
  of a finite group}\label{Sec:GdSj}

\indent
In this section,
we show that if $G$ is a finite group
then $C^* (G),$ with its natural $G$-grading,
is \sj\  in the graded sense.
This is an application of \Lem{L:A7},
the same result that played a key role in the proof that
\fd\  \ca s are \eqsj.

Presumably much more general results are possible.
Indeed, the appropriate setting may be
actions of \fd\  Hopf algebras or compact quantum groups on \fd\  \ca s.

The following definition is a special case of Definitions 3.1 and~3.4
of~\cite{Ex},
of a \ca\  (topologically) graded by a discrete group~$G.$
In~\cite{Ex},
the group is not necessarily finite,
and one only requires that $\bigoplus_{g \in G} A_g$
be dense in~$A.$
Continuity of the projection to~$A_1$
(as in Definition~3.4 of~\cite{Ex})
is automatic when the group is finite
and $\bigoplus_{g \in G} A_g = A.$

\begin{dfn}\label{D:Graded}
Let $G$ be a finite group,
and let $A$ be a \ca.
A {\emph{$G$-grading}} on~$A$ is a direct sum decomposition as
Banach spaces
\[
A = \bigoplus_{g \in G} A_g
\]
such that if $g, h \in G,$ $a \in A_g,$ and $b \in A_h,$
then $a b \in A_{g h}$ and $a^* \in A_{g^{-1}}.$
(We do not say anything about the direct sum norm except that
it is equivalent to the usual norm on~$A.$)

A subspace $E \subset A$ is {\emph{graded}} if
$E = \sum_{g \in G} (E \cap A_g).$

We denote by $P_g,$ or $P_g^A,$
the projection map from $A$ to $A_g$ associated with this
direct sum decomposition.
\end{dfn}

To put this definition in context,
we make three remarks.
First,
when $G$ is finite,
a $G$-grading of~$A$ is the same as an identification
of~$A$ with the \ca\  of a Fell bundle over~$G.$
The basic correspondence is given in VIII.16.11 and VIII.16.12
of~\cite{FD2},
but in general it is not bijective.
It is bijective for Fell bundles over discrete groups
which are amenable
in the sense of Definition~4.1 of~\cite{Ex},
and in particular for all Fell bundles and topological gradings when
$G$ is amenable.
(This follows from Theorem~4.7 of~\cite{Ex}.)
Since our groups are finite,
the correspondence is bijective in our case.

Second, for discrete groups~$G,$
a normal coaction on a \ca~$A$
(as defined before Definition~1.1 of~\cite{Qg})
is the same as an identification
of~$A$ with the \ca\  of a Fell bundle over~$G.$
See Proposition~3.3 and Theorem~3.8 of~\cite{Qg}.

Finally,
if $G$ is abelian,
then a $G$-grading on~$A$ is the same as an action
$\af \colon {\widehat{G}} \to \Aut (A).$
Given a $G$-grading on~$A$
and $\ta \in {\widehat{G}},$
we define $\af_{\ta} \in \Aut (A)$
by $\af_{\ta} (a) = \ta (g) a$ for $a \in A_g.$
Given~$\af,$
for $g \in G$ we set
\[
A_g = \big\{ a \in A \colon
 {\mbox{$\af_{\ta} (a) = \ta (g) a$
        for all $\ta \in {\widehat{G}}$}} \big\},
\]
that is, $A_g$ is the spectral subspace for $g$
when $g$ is regarded as an element of the second dual of~$G.$

\begin{rmk}\label{R:PjRmk}
Let $G$ be a finite group,
and let $A = \bigoplus_{g \in G} A_g$ be a $G$-grading of~$A.$
Then the summand $A_1$ is a \ca.
(This is clear.)
Let $P_g \colon A \to A_g$ be as in Definition~\ref{D:Graded}.
Then $P_1$ is a conditional expectation onto~$A_1,$
and $\| P_g \| \leq 1$ for all $g \in G.$
(See Theorem~3.3 and Corollary~3.5 of~\cite{Ex}.)
\end{rmk}

\begin{dfn}\label{D:GradedHom}
Let $G$ be a finite group,
let $A$ and $B$ be a \ca s
with $G$-gradings $A = \bigoplus_{g \in G} A_g$
and $B = \bigoplus_{g \in G} B_g,$
and let $\ph \colon A \to B$ be a \hm.
We say that $\ph$ is {\emph{graded}}
if for every $g \in G$ we have $\ph (A_g) \subset B_g.$
\end{dfn}

\begin{rmk}\label{R:GradedQuot}
Let $G$ be a finite group,
let $A$ be a \ca\  %
with $G$-grading $A = \bigoplus_{g \in G} A_g,$
and let $I \subset A$ be a graded ideal.
Then $A / I$ becomes a graded \ca\  with the grading
\[
(A / I)_g = A_g / (A_g \cap I) = (A_g + I) / I,
\]
and the quotient map $A \to A / I$ is a graded \hm.
\end{rmk}

\begin{rmk}\label{R:GradedLim}
Let $G$ be a finite group.
Then the direct limit of a direct system of $G$-graded \ca s
with graded maps is a $G$-graded \ca\  in an obvious way.
\end{rmk}

\begin{rmk}\label{R:GradedCrPrd}
Let $G$ be a finite group,
let $A$ be a \ca,
and let $\af \colon G \to \Aut (A)$ be an action of $G$ on~$A.$
Then the crossed product $C^* (G, A, \af)$
is graded in the following way.
Let $u_g \in C^* (G, A, \af)$
(or in $M ( C^* (G, A, \af) )$ if $A$ is not unital)
be the standard unitary corresponding to $g \in G.$
Then
\[
C^* (G, A, \af)_g = \{ a u_g \colon a \in A \}.
\]
(This is the dual coaction.)
\end{rmk}

\begin{rmk}\label{R:GradedGpCst}
In Remark~\ref{R:GradedCrPrd},
take $A = \C$ and take $\af$ to be the trivial action.
This gives a canonical $G$-grading on $C^* (G).$
If $u_g \in C^* (G)$
is the unitary corresponding to $g \in G,$
then $C^* (G)_g = \C u_g.$
\end{rmk}

The following definition is the analog of
Definition 14.1.3 of~\cite{Lr}.

\begin{dfn}\label{D:GdSj}
Let $G$ be a finite group,
and let $A$ be a \ca\  %
with $G$-grading $A = \bigoplus_{g \in G} A_g.$
We say that the grading is {\emph{graded semiprojective}}
if whenever $C$ is a a \ca\  %
with $G$-grading $C = \bigoplus_{g \in G} C_g,$
$J_0 \subset J_1 \subset \cdots$ are graded ideals in~$C,$
$J = {\overline{\bigcup_{n = 0}^{\infty} J_n}},$
and $\ph \colon A \to C / J$
is a graded \hm,
then there exists~$n$
and a graded \hm\  $\ps \colon A \to C / J_n$
such that the composition
\[
A \stackrel{\ps}{\longrightarrow} C / J_n \longrightarrow C / J
\]
is equal to~$\ph.$

When no confusion can arise, we say that $A$ is graded semiprojective.
\end{dfn}

Here is the diagram:
\[
\xymatrix{
& C \ar@{-->}[d] \ar@/^2pc/[dd] \\
& C / J_n \ar@{-->}[d] \\
A \ar[r]_-{\ph} \ar@{-->}[ru]^{\ps} & C / J.
}
\]

\begin{thm}\label{T:CStG}
Let $G$ be a finite group.
Then $C^* (G),$
with the $G$-grading in Remark~\ref{R:GradedGpCst},
is graded \sj.
\end{thm}

\begin{proof}
Set $\ep_0 = \frac{1}{6 \cdot 34},$
and choose $\ep > 0$ such that $\ep \leq \ep_0$ and
such that whenever $A$ is a \uca,
$u \in U (A),$
and $a \in A$ satisfies $\| a - u \| < \ep,$
then we have
$\big\| a (a^* a)^{-1/2} - u \big\| < \ep_0.$

Let the notation be as in Definition~\ref{D:GdSj}
and Remark~\ref{R:EqSj}(\ref{R:EqSj:Nt}).
Further, for  $g \in G$ let $P_g^{(n)} \colon C/J_n \to (C/J_n)_g$
and $P_g \colon C/J \to (C/J)_g$ be the \pj\  maps
associated to the gradings, as in \Def{D:Graded}.
Also let $u_g \in C^* (G)$ be the unitary associated with the
group element $g \in G,$ as in Remark~\ref{R:GradedGpCst}.

Let $\ph \colon A \to C / J$ be a unital graded \hm.
Since \fd\  \ca s are \sj,
there exist $n_0$ and a unital \hm\  %
(not necessarily graded) $\ps_0 \colon A \to C / J_{n_0}$
which lifts~$\ph.$
Since $\pi_{n_0}$ and $\ph$ are graded, for $g \in G$ we have
\[
\pi_{n_0} \big( \ps_0 ( u_g) - P_g^{(n_0)} ( \ps_0 (u_g)) \big)
  = \ph (u_g) - P_g (\ph (u_g))
  = 0.
\]
Therefore there exists $n \geq n_0$ such that
for all $g \in G$
we have
\begin{equation}\label{Eq:TGdSj-1}
\big\| \pi_{n, n_0} \big( \ps_0 ( u_g)
              - P_g^{(n_0)} ( \ps_0 (u_g)) \big) \big\|
  < \ep.
\end{equation}
Set $\ps_1 = \pi_{n, n_0} \circ \ps_0$
and for $g \in G$ set $c_g = P_g^{(n)} ( \ps_1 (u_g)),$
which is in $(C / J)_g.$
Then~(\ref{Eq:TGdSj-1}) becomes
\[
\| \ps_1 ( u_g) - c_g \| < \ep
\]
for all $g \in G.$

Since $\ep < 1,$
we can
define $\rh_0 \colon G \to U (C / J_n)$
by $\rh_0 (g) = c_g (c_g^* c_g)^{- 1/2}.$
Since $c_g \in (C / J)_g,$
we have $c_g^* c_g \in (C / J)_1,$
whence $(c_g^* c_g)^{- 1/2} \in (C / J)_1,$
so that $\rh_0 (g) \in (C / J)_g.$
Moreover, the choice of $\ep$ ensures that
$\| c_g - \rh_0 (g) \|
  < \ep_0.$
Therefore
\[
\| \rh_0 (g) - \ps_1 (u_g) \|
 \leq \| \rh_0 (g) - c_g \| + \| c_g - \ps_1 (u_g) \|
 < \ep_0 + \ep
 \leq 2 \ep_0.
\]
Let $g, h \in G.$
Since
\[
\ps_1 (u_g), \, \ps_1 (u_h) \in U (C / J_n)
\andeqn
\ps_1 (u_{g h}) = \ps_1 (u_g) \ps_1 (u_h),
\]
it follows that
\[
\| \rh_0 (g h) - \rh_0 (g) \rh_0 (h) \| < 6 \ep_0.
\]
Since $\pi_n (c_g) = \ph (u_g)$ is unitary,
we also get
\[
\pi_n (\rh_0 (g)) = \pi_n (c_g) = \ph (u_g).
\]

Inductively define functions $\rh_m \colon G \to U (C / J_n)$
by (following \Lem{L:A7})
\[
\rh_{m + 1} (g)
 = \exp \left( \frac{1}{\card (G)} \sum_{h \in G}
      \log \Big( \rh_m (h)^* \rh_m (h g) \rh_m (g)^* \Big)
       \right) \rh_m (g)
\]
for $g \in G.$
Since $6 \ep_0 = \frac{1}{34} < \frac{1}{17},$
Lemma~\ref{L:A7} implies that the functions $\rh_m$
are well defined maps $\rh_m \colon G \to U (C / J_n)$
such that $\rh (g) = \limi{m} \rh_m (g)$
defines a \hm\  $\rh \colon G \to U (C / J_n)$
satisfying
\[
\| \rh (g) - \rh_0 (g) \|
 \leq \frac{2 \cdot 6 \ep_0}{1 - 17 \cdot 6 \ep_0}
\andeqn
\pi_n ( \rh (g)) = \ph (u_g)
\]
for all $g \in G.$

We claim that $\rh_m (g) \in (C / J_n)_g$
for all $g \in G$ and $m \in \Nz.$
The proof is by induction on~$m.$
We know this is true for $m = 0.$
Assume it is true for~$m.$
For all $g, h \in G$ we have
\[
\rh_m (h)^* \rh_m (h g) \rh_m (g)^*
 \in (C / J_n)_h^* (C / J_n)_{h g} (C / J_n)_g^*
 \subset (C / J_n)_1.
\]
Therefore also
\[
\exp \left( \frac{1}{\card (G)} \sum_{h \in G}
      \log \Big( \rh_m (h)^* \rh_m (h g) \rh_m (g)^* \Big)
       \right)
   \in (C / J_n)_1,
\]
and the induction step follows.
This proves the claim.
Taking limits, we get $\rh (g) \in (C / J_n)_g$ for all $g \in G.$

By the universal property of $C^* (G),$ there is a unital
\hm\  $\ps \colon C^* (G) \to C / J$
such that $\ps (u_g) = \rh (g)$ for all $g \in G.$
By construction, $\ps$ is graded.
Moreover, $\pi_n \circ \ps (u_g) = \ph (u_g)$ for all $g \in G,$
so the universal property of $C^* (G)$
implies that $\pi_n \circ \ps = \ph.$
Thus $\ps$ lifts~$\ph.$
\end{proof}

\end{document}